\documentclass[11pt]{article}
\usepackage{pb-diagram}
\usepackage{amsmath,amsthm,amsfonts,amssymb,latexsym}
\usepackage{amsmath}
\usepackage{hyperref}
\usepackage{a4wide}
\usepackage[all]{xy}

\newtheorem{thm}{Theorem}[section]
\newtheorem{cor}[thm]{Corollary}
\newtheorem{lem}[thm]{Lemma}
\newtheorem{prop}[thm]{Proposition}
\newtheorem{ex}[thm]{Example}
\newtheorem{defn}[thm]{Definition}

\newtheorem{rmk}[thm]{Remark}

\def\a{\alpha}
\def\b{\beta}

\def\m{\mu}

\begin{document}

\title{Construtions and bimodules of BiHom-alternative and BiHom-Jordan algebras}
\author{Sylvain Attan \thanks{D\'{e}partement de Math\'{e}matiques, Universit\'{e} d'Abomey-Calavi
01 BP 4521, Cotonou 01, B\'{e}nin. E.mail: syltane2010@yahoo.fr}\and Ismail Laraiedh \thanks{Departement of Mathematics, Faculty of Sciences, Sfax University, BP 1171, 3000 Sfax, Tunisia. E.mail:
Ismail.laraiedh@gmail.com and Departement of Mathematics, College of Sciences and Humanities - Kowaiyia, Shaqra University,
Kingdom of Saudi Arabia. E.mail:
ismail.laraiedh@su.edu.sa}}

\maketitle
\begin{abstract}
The aim of this paper is to give some construction results and examples of BiHom-alternative and BiHom-Jordan algebras.
Next, we define BiHom-alternative and BiHom-Jordan (bi)modules and we prove that from a given BiHom-alternative and BiHom-Jordan bimodules, a sequence of this kind of bimodules can be constructed. Some relations between Bihom-associative, BiHom-alternative and BiHom-Jordan  bimodules are given.

\end{abstract}
{\bf 2010 Mathematics Subject Classification:} 17A30, 17B10, 17C50, 17D15.

{\bf Keywords:} BiHom-alternative algebras, BiHom-Jordan algebras, BiHom-modules, Bimodules.
\section{Introduction}
As generalization of associative algebras, alternative algebras are algebras whose associator is an alternating function. Jordan algebras is another important class of algebras and were introduced in Physics context by
Pascual Jordan in 1933 to provide a new formalism for Quantum Mechanics
and developed from the algebraic viewpoint by Jacobson in \cite{Jacob3}.
 As associative algebras are related to Jordan algebras, alternative algebras are closely related to Jordan algebras \cite{fgcht, pjjnew, so, tasfdv}.
The study of bimodule (or representation) of Jordan algebras was initiated by N. Jacobson \cite{Jacob1} and the one of alternative algebras was considered by Schafer \cite{Schaf1}.

Algebras where the identities defining the structure are twisted by a homomorphism are called Hom-algebras.
They started from Hom-Lie algebras introduced and discussed in
\cite{HAR1, dlsds1, dlsds2, dlsds3,HomHopf}, motivated by quasi-deformations of Lie algebras of vector fields, in particular q-deformations of Witt and Virasoro algebras.
Hom-associative algebras were introduced in \cite{MAK3} while Hom-alternative and Hom-Jordan algebras are introduced in \cite{MAK1, YAU3}
as twisted generalizations of alternative and Jordan algebra respectively.
Dualizing  Hom-associative algebras, one can define Hom-coassociative
coalgebras, Hom-bialgebras and Hom-Hopf algebras which were introduced in
\cite{HomAlgHomCoalg,HomHopf}, see also \cite{scig,Yau:ClassicYangBaxter,Yau:HomQuantumGrp1,Yau:HomQuantumGrp2}.

 (Bi)modules over an ordinary algebra has been extended to those of
Hom-algebras in many works \cite{ YS, YAU4, YAU2}, in particular Hom-Jordan and Hom-alternative bimodules are introduced and studied in \cite{attan1, attan2}.

In \cite{scig} the authors gave a new look to Hom-type algebras from a category theoretical point of view. A generalization of this approach
led the authors of \cite{GRAZIANI} to introduce BiHom-algebras, which are algebras where the identities defining the structure are twisted by two homomorphisms $\alpha$ and $\beta$.
This class of algebras can be viewed as an extension of the class of Hom-algebras since, when the two linear maps of a BiHom-algebra are the same, it reduces to a Hom-algebra. These
algebraic structures include BiHom-associative algebras, BiHom-Lie algebras and BiHom-bialgebras. More applications of BiHom-algebras, BiHom-Lie superalgebras, BiHom-Lie colour
algebras and BiHom-Novikov algebras can be
found in \cite{elkadri, A S O, S S, Guo1, chtioui1,chtioui2, Liu1, Liu2}.

The aim of this paper is to extend the work done in \cite{attan1} in BiHom-algebras setting. More precisely, we introduce
BiHom-alternative bimodules  and BiHom-Jordan bimodules  and  discuss about some findings. The paper is
organized as follows. Section 2 contains some necessary important basic notions and notations related to BiHom-algebras, BiHom-alternative algebras, BiHom-Jordan algebras
and modules over BiHom-associative algebras.
Section 3 presents some useful methods for constructions of BiHom-alternative algebras and BiHom-Jordan algebras. In section 4,
we give definitions and some properties of BiHom-alternative and BiHom-Jordan (bi)modules and we prove that from a given
BiHom-alternative algebra and BiHom-Jordan algebra bimodules, a sequence of this kind of bimodules can be constructed.

Throughout this paper $\mathbb{K}$ is an algebraically closed field of characteristic $0$ and $A$ is a $\mathbb{K}$-vector space.
\section{Preliminaries}
This section contains necessary important basic notions and notations which will be used in next sections.
For the map $\mu: A^{\otimes 2}\longrightarrow A,$ we will sometimes $\mu(a\otimes b)$ as $\mu(a,b)$ or $ab$  for $a,b\in A$ and if $V$ is another vector space, $\tau_1: A\otimes V\longrightarrow V\otimes A$ (resp. $\tau_2: V\otimes A\longrightarrow A\otimes V$) denote the twist isomorphism $\tau_1(a\otimes v)=v\otimes a$ (resp.
$\tau_2(v\otimes a)=a\otimes v$).
\begin{defn}
A BiHom-module is a pair $(M,\alpha_M,\beta_M)$ consisting of a $\mathbb{K}$-module $M$ and
a linear self-maps $\alpha_M, \beta_M: M\longrightarrow M$  such that $\alpha_M\beta_M=\beta_M\alpha_M.$ A morphism
$f: (M,\alpha_M, \beta_M)\rightarrow (N,\alpha_N,\beta_N)$ of BiHom-modules is a linear map
 $f: M\longrightarrow N$ such that $f\alpha_M=\alpha_N f$ and
 $f\beta_M=\beta_N f.$
\end{defn}
\begin{defn}\cite{GRAZIANI}
A BiHom-algebra is a quadruple $(A,\mu,\alpha,\beta)$ in which $(A,\alpha,\beta)$ is a BiHom-module, $\mu : A^{\otimes 2} \rightarrow A$ is a linear map.
The BiHom-algebra $(A,\mu,\alpha;\beta)$ is said to be  multiplicative if $\alpha\circ\mu=\mu\circ\alpha^{\otimes 2}$ and $\beta\circ\mu=\mu\circ\beta^{\otimes 2}$ (multiplicativity).
\end{defn}

\begin{defn}Let $(A,\mu,\alpha,\beta)$ be a BiHom-algebra. Then
\begin{enumerate}
\item
A BiHom-subalgebra of $(A,\mu,\alpha,\beta)$ is a linear subspace $H$ of $A$, which is closed for the multiplication $\mu$ and invariant by $\alpha$ and $\beta$, that is, $\mu(x,y)\in H,~\alpha(x)\in H$ and $\beta(x)\in H$ for all $x,y\in H$. If furthermore $\mu(a,b)\in H$ and $\mu(b,a)\in H$ for all $(a,b)\in A\times H,$ then $H$ is called a two-sided BiHom ideal of $A$.
  \item $(A,\mu,\alpha,\beta)$ is said to be regular if $\alpha$ and $\beta$ are algebra automorphisms.
  \item $(A,\mu,\alpha,\beta)$ is said to be involutive if $\alpha$ and $\beta$ are two involutions,  that is $\alpha^2=\beta^2=id$.
\end{enumerate}
\end{defn}
\begin{defn}
Let $(A,\mu,\alpha,\beta)$ and $(A^{'},\mu^{'},\alpha^{'},\beta^{'})$ be two BiHom- algebras. Then
a homomorphism $f:A\longrightarrow A^{'}$ is said to be a BiHom-algebra morphism  if the following conditions hold:
$$\begin{array}{llll}&&f  \circ \mu= \mu^{'} \circ(f \otimes f),\\&&f\circ \alpha=\alpha^{'} \circ f,\\&&f\circ \beta=\beta^{'}\circ  f.\end{array} $$
\end{defn}
Denote by $\Gamma_{f}=\{x+f(x);~~x\in A\}\subset A\oplus A'$  the graph of a linear map $f:A\longrightarrow A^{'}$.
\begin{defn} \cite{Liu1} Let $(A, \mu, \a,\b)$ be a BiHom-algebra and let $\lambda\in \mathbb{K}$. Let $R: A\rightarrow A$ be a linear map satisfying
\begin{eqnarray*}
\mu(R(x), R(y))=R(\mu(R(x), y)+\mu(x, R(y))+\lambda \mu(x, y)), ~~~\mbox{for any $x,y\in A$}.
\end{eqnarray*}
Then $R$ is called  a Rota-Baxter operator of weight $\lambda$ and $(A, \mu, \a,\b, R)$ is called a Rota-Baxter
BiHom-algebra of weight $\lambda$.
\end{defn}
\begin{defn}\cite{GRAZIANI}
Let $(A,\mu,\alpha,\beta)$ be a BiHom-algebra.
A BiHom-associator of $A$  is the trilinear map $as_{\alpha,\beta}:A^{\otimes 3} \longrightarrow A$ defined by
\begin{equation}\label{BiHomAss}
    as_{\alpha,\beta}=\mu \circ (\mu\otimes \beta- \alpha \otimes \mu).
\end{equation}
In terms of elements, the map $as_{\alpha,\beta}$ is given by
$$ as_{\alpha,\beta}(x,y,z)=\mu(\mu(x,y),\beta(z))-\mu(\alpha(x),\mu(y,z)),\ \forall x,y,z \in A.$$
\end{defn}
Note that if  $\alpha=\beta=id $, then the BiHom-associator coincide with the usual associator denoted by $as(,,).$
\begin{defn}
A BiHom-associative algebra \cite{GRAZIANI} is a multiplicative Bihom-algebra $(A,\mu,\alpha,\beta)$
satisfying the following BiHom-associativity condition:
\begin{equation}\label{BiHom ass identity}
    as_{\alpha,\beta}(x,y,z)=0,\ \text{for all}\ x,y,z \in A.\
\end{equation}
\end{defn}

Clearly, a Hom-associative algebra $(A,\mu,\alpha)$ can be regarded as a BiHom-associative
algebra $(A,\mu,\alpha,\alpha)$.

\vspace{0.5cm}
 An important class of BiHom-algebras that is considered here is the one of BiHom-alternative algebras.
These algebras have been introduced By Chtioui et al in \cite{chtioui1}.
\begin{defn}
A left BiHom-alternative algebra $($resp. right BiHom-alternative algebra$)$ is a multiplicative Bihom-algebra $(A,\mu,\alpha,\beta)$  satisfying the left BiHom-alternative identity,
\begin{equation}\label{sa1}
as_{\alpha,\beta}(\beta(x),\alpha(y),z)+as_{\alpha,\beta}(\beta(y),\alpha(x),z)=0,
\end{equation}
respectively, the right BiHom-alternative identity,
\begin{equation}\label{sa2}
as_{\alpha,\beta}(x,\beta(y),\alpha(z))+as_{\alpha,\beta}(x,\beta(z),\alpha(y))=0,
\end{equation}
 for all $x,y,z \in A$.
A BiHom-alternative algebra is the one which is both a left and right BiHom-alternative algebra.
\end{defn}
\begin{rmk}
Any BiHom-associative algebra is a BiHom-alternative algebra.
\end{rmk}

\begin{lem}\cite{chtioui1} Let $(A,\mu,\alpha,\beta)$ be a multiplicative BiHom-algebra. Then
\begin{enumerate}
  \item  $A$ is a left BiHom-alternative algebra if and only if
\begin{equation}\label{is1}as_{\alpha,\beta}(\beta(x),\alpha(x),y)=0 \mbox{ for all $x,y \in A$} \end{equation}
  \item $A$ is a right BiHom-alternative algebra if and only if
\begin{equation}\label{is2}as_{\alpha,\beta}(x,\beta(y),\alpha(y))=0  \mbox{ for all $x,y \in A$} \end{equation}
\end{enumerate}
\end{lem}
\begin{proof}
 Follows by a direct computation that is left to the reader.
\end{proof}
Observe that when $\alpha=\beta=id$ the left BiHom-alternative identity $(\ref{sa1})$ (resp.  right BiHom-alternative identity $(\ref{sa2}$))
reduces to  left alternative identity (resp.  right alternative identity).

\vspace{0.5cm}
As BiHom-alternative algebras, BiHom-Jordan algebras are fundamental objects of this paper. They appear as cousins
of BiHom-alternative algebras and these two BiHom-algebras are related as (Hom-)Jordan and (Hom-)alternative algebras.
\begin{defn}\label{BiHom jordan}\cite{chtioui1}
A BiHom algebra $(A,\mu,\alpha,\beta)$ is called a BiHom-Jordan algebra if:
\begin{enumerate}
  \item $\mu \Big(\beta(x),\alpha(y)\Big)=\mu\Big(\beta(y),\alpha(x)\Big),\ \forall x,y \in A$ (BiHom-commutativity condition),
  \item $as_{\alpha,\beta}\Big(\mu\big(\beta^2(x),\alpha\beta(x)\big),\alpha^2\beta(y),\alpha^3(x)\Big)=0,\ \forall x,y \in A$(BiHom-Jordan identity).
\end{enumerate}
\end{defn}
Note that if $\alpha=\beta=id$, we obtain a Jordan algebra. Then we conclude that if $(A,\mu)$ is a Jordan algebra, $(A,\mu,id,id)$ can be viewed as a BiHom-Jordan algebra.
\begin{rmk}\cite{chtioui1}
The BiHom-Jordan identity can be writen as:  $$\circlearrowleft_{x,w,z}as_{\alpha,\beta}\Big(\mu\big(\beta^2(x),\alpha\beta(w)\big),\alpha^2\beta(y),\alpha^3(z)\Big)=0,
~\forall x,y,z,w\in A,$$
where $\circlearrowleft_{x,w,z}$ denotes the summation over the cyclic permutation on $x,w,z.$
\end{rmk}
\begin{defn} Let $(A,\mu,\alpha,\beta)$ be any BiHom-algebra.
\begin{enumerate}
\item A BiHom-module $(V,\phi,\psi)$ is called an $A$-bimodule if it comes equipped with a left and a right structures maps on $V$ that is morphisms
$\rho_{l}: (A\otimes V, \alpha\otimes\phi,\beta\otimes\psi)\rightarrow (V,\phi,\psi),~a\otimes v\mapsto a.v$ and $\rho_{r}:(V\otimes A,\phi\otimes\alpha,\psi\otimes\psi)\rightarrow (V,\phi,\psi),~v\otimes a\mapsto v.a$ of Bihom-modules.
\item
A morphism $f:(V,\phi,\psi,\rho_{l},\rho_{r})\rightarrow (W,\phi',\psi',\rho_{l}',\rho_{r}')$ of $A$-bimodules is a morphism of the underlying BiHom-modules such that
$$f\circ\rho_{l}=\rho_{l}'\circ(Id_{A}\otimes f)~~and~~f\circ\rho_{r}=\rho_{r}'\circ( f \otimes Id_{A}).$$
That yields the commutative diagrams
  $$
\xymatrix{
 A\otimes V \ar[d]_{  Id_{A}\otimes f}\ar[rr]^{\rho_l}
                && V  \ar[d]^{f}  \\
 A\otimes W \ar[rr]^{\rho_l'}
                && W}\quad\xymatrix{
 V\otimes A \ar[d]_{ f\otimes Id_{A} }\ar[rr]^{\rho_r}
                && V  \ar[d]^{f}  \\
 W\otimes A \ar[rr]^{\rho_r'}
                && W  }
$$
\item Let $(V,\phi,\psi)$ be an $A$-bimodule with structure maps $\rho_l$ and $\rho_r$. Then the module BiHom-associator of $V$ is a trilinear map $as_{V_{\phi,\psi}}$ defined as:
\begin{eqnarray}
as_{V_{\phi,\psi}}\circ Id_{V\otimes A\otimes A}=\rho_r\circ(\rho_r\otimes\beta)-
\rho_r\circ(\phi\otimes\mu)\nonumber\\
as_{V_{\phi,\psi}}\circ Id_{A\otimes V\otimes A}=\rho_r\circ(\rho_l\otimes\beta)-
\rho_l\circ(\alpha\otimes\rho_r)\nonumber\\
as_{V_{\phi,\psi}}\circ Id_{A\otimes A\otimes V}=\rho_l\circ(\mu\otimes\psi)-
\rho_l\circ(\alpha\otimes\rho_l)~\nonumber
\end{eqnarray}
\end{enumerate}
\end{defn}
\begin{rmk}
\begin{enumerate}
\item If $\alpha=\beta=Id_A$ and $\phi=\psi=Id_V$, then $as_{V_{\phi,\psi}}$ is denoted by $as_V.$
\item If $\beta=\alpha$ and $\psi=\phi$ then, the module BiHom-associator $as_{V_{\phi,\psi}}$ reduces to the module Hom-associator $as_{A,V}$ given in \cite{attan1}.
\end{enumerate}
\end{rmk}
Now, let consider the following notions for BiHom-associative algebras.
\begin{defn}\cite{GRAZIANI}
Let $(A,\mu,\alpha,\beta)$ be a BiHom-associative algebra and $(V, \phi,\psi)$ be a BiHom-module. Then
\begin{enumerate}
\item A left BiHom-associative $A$-module structure on $V$ consists of a morphism $\rho_{l}: A\otimes V\longrightarrow V$ of BiHom-modules, such that
\begin{eqnarray}\label{00}
\rho_{l}\circ(\alpha\otimes\rho_{l})=\rho_{l}\circ(\mu\otimes\psi).\label{LeftAssMod}
\end{eqnarray}
In terms of diagram, we have
  $$
\xymatrix{
 A\otimes A\otimes V \ar[d]_{ \mu\otimes\psi}\ar[rr]^{\alpha\otimes\rho_{l}}
                && A\otimes V  \ar[d]^{\rho_{l}}  \\
 A\otimes V \ar[rr]^{\rho_{l}}
                && V  }$$
\item A right BiHom-associative $A$-module structure on $V$ consists of a morphism $\rho_{r}: V\otimes A\longrightarrow V$ of BiHom-modules, such that
\begin{eqnarray}\label{11}
\rho_{r}\circ(\phi\otimes\mu)=\rho_{r}\circ(\rho_{r}\otimes\beta).\label{RightAssMod}
\end{eqnarray}
In terms of diagram, we have
 $$
\xymatrix{
V\otimes A\otimes A  \ar[d]_{ \phi\otimes\mu}\ar[rr]^{\rho_{r}\otimes\beta}
                && V\otimes A  \ar[d]^{\rho_{r}}  \\
 V\otimes A \ar[rr]^{\rho_{r}}
                && V  }$$
\item A BiHom-associative $A$-bimodule structure on $V$ consists of two structure maps $\rho_l: A\otimes V\longrightarrow V$ and $\rho_r: V\otimes A\longrightarrow V$
such that $(V, \phi,\psi,\rho_l)$ is
a left BiHom-associative $A$-module, $(V, \phi,\psi,\rho_r)$ is a right BiHom-associative $A$-module and that the following BiHom-associativity condition holds:
\begin{eqnarray}
\rho_l\circ(\alpha\otimes\rho_r)=\rho_r\circ(\rho_l\otimes\beta).\label{Assoc1}
\end{eqnarray}
In terms of diagram, we have
 $$
\xymatrix{
A\otimes V\otimes A  \ar[d]_{ \alpha\otimes\rho_{r}}\ar[rr]^{\rho_{l}\otimes\beta}
                && V\otimes A  \ar[d]^{\rho_{r}}  \\
A\otimes V \ar[rr]^{\rho_{l}}
                && V  }$$
\end{enumerate}
\end{defn}
\section{ Constructions of BiHom-alternative and  BiHom-Jordan algebras}
In this section, we provide
some constructions results of BiHom-alternative algebras and BiHom-Jordan algebras. Some examples are also given.


\subsection{Constructions of BiHom-alternative algebras}

\begin{prop}
Let $(A,\mu,\alpha,\beta)$ be a BiHom-alternative algebra and $I$ be a two-sided BiHom-ideal of $(A,\mu,\alpha,\beta)$.
Then $(A/I,\overline{\mu},\overline{\alpha},\overline{\beta})$ is a BiHom-alternative algebra where $\overline{\mu}(\overline{x},\overline{y})=\overline{\mu(x,y)}$,
$\overline{\alpha}(\overline{x})=\overline{\alpha(x)}$  and $\overline{\beta}(\overline{x})=\overline{\beta(x)}$, for all $\overline{x},\overline{y}\in A/I$
\end{prop}
\begin{proof}
First, note that the multiplicativity of $\overline{\alpha}$ and $\overline{\beta}$  with respect to $\overline{\mu}$, follows from the one of $\alpha$ and $\beta$ with respect to $\mu$.
Next, pick $\overline{x},\overline{y}\in A/I$. Then the left and right BiHom-alternative identity in $(A/I,\overline{\mu},\overline{\alpha},\overline{\beta})$ are proved as follows
$$\begin{array}{llllllll}as_{\overline{\alpha},\overline{\beta}}(\overline{\beta}(\overline{x}),\overline{\alpha}(\overline{x}),\overline{y})&=&\overline{\mu}(\overline{\mu}(\overline{\beta}(\overline{x}),\overline{\alpha}(\overline{x})),\overline{\beta}(\overline{y}))
-\overline{\mu}(\overline{\alpha}\overline{\beta}(\overline{x}),\overline{\mu}(\overline{\alpha}(\overline{x}),\overline{y}))\\
&=&\overline{\mu(\mu(\beta(x),\alpha(x)),\beta(y))-\mu(\alpha\beta(x),\mu(\alpha(x),y))}\\
&=&\overline{as_{\alpha,\beta}(\beta(x),\alpha(x),y)}=\overline{0} ~~(\textsl{by ~(\ref{is1})~ in}~ A)
\end{array}$$
$$\begin{array}{llllllll}as_{\overline{\alpha},\overline{\beta}}(\overline{x},\overline{\beta}(\overline{y}),\overline{\alpha}(\overline{y}))&=&\overline{\mu}(\overline{\mu}(\overline{x},\overline{\beta}(\overline{y})),\overline{\alpha}\overline{\beta}(\overline{y}))
-\overline{\mu}(\overline{\alpha}(\overline{x}),\overline{\mu}(\overline{\beta}(\overline{y}),\overline{\alpha}(\overline{y})))\\
&=&\overline{\mu(\mu(x,\beta(y)),\alpha\beta(y))-\mu(\alpha(x),\mu(\beta(y),\alpha(y)))}\\
&=&\overline{as_{\alpha,\beta}(x,\beta(y),\alpha(y))}=\overline{0} ~~(\textsl{by ~(\ref{is2})~ in}~ A).
\end{array}$$
Then $(A/I,\overline{\mu},\overline{\alpha},\overline{\beta})$ is a BiHom-alternative algebra.
\end{proof}
\begin{prop}\label{bylv}
 Given two BiHom-alternative algebras $(A,\mu_A,\alpha_{A},\beta_{A})$ and $(B,\mu_B,\alpha_{B},\beta_{B}),$ there is a BiHom-alternative algebra $(A\oplus B, \mu_{A\oplus B}, \alpha_{A}+\alpha_{B},\beta_{A}+\beta_{B}),$ where the bilinear map $\mu_{A\oplus B}:(A\oplus B)^{\times 2}\longrightarrow (A\oplus B)$ is given by
 $$\begin{array}{llll}
 \mu_{A\oplus B}(a_1+b_1,a_2+b_2)=\mu_A(a_1,a_2)+\mu_B(b_1,b_2), \forall \ a_1,a_2\in A,\ \forall \ b_1,b_2\in B,\nonumber
 \end{array}$$
 and the linear  maps $\beta=\beta_{A}+\beta_{B},~\alpha=\alpha_{A}+\alpha_{B}: (A\oplus B)\longrightarrow (A\oplus B)$ are given by
$$ \begin{array}{lll}
(\alpha_{A}+\alpha_{B})(a+b)&=& \alpha_{A}(a)+\alpha_{B}(b),\\
(\beta_{A}+\beta_{B})(a+b)&=& \beta_{A}(a)+\beta_{B}(b),~~ \forall \ (a,b)\in A\times B.
 \end{array}$$
 \end{prop}
\begin{proof}
Since $\alpha_{A}\circ\beta_{A}=\beta_{A}\circ\alpha_{A}$ and $\alpha_{B}\circ\beta_{B}=\beta_{B}\circ\alpha_{B}$, we have
$(\alpha_{A}+\alpha_{B})\circ(\beta_{A}\circ\beta_{B})=(\beta_{A}\circ\beta_{B})\circ(\alpha_{A}+\alpha_{B})$.

First, $(\alpha_{A}+\alpha_{B})$ is multiplicative with respect to $\mu_{A\oplus B}.$ Indeed, $\forall (a_{1},b_{1}),~(a_{2},b_{2})\in A\times B$
$$\begin{array}{lllllll}
&&(\alpha_{A}+\alpha_{B})\circ\mu_{A\oplus B}(a_1+b_1,a_2+b_2)\\&=&(\alpha_{A}+\alpha_{B})(\mu_A(a_1,a_2)+\mu_B(b_1,b_2))\\
&=&\alpha_{A}\circ\mu_A(a_1,a_2)+\alpha_{B}\circ\mu_B(b_1,b_2)\\
&=&\mu_A(\alpha_{A}(a_1),\alpha_{A}(a_2))+\mu_B(\alpha_{B}(b_1),\alpha_{B}(b_2)\\
&=&\mu_{A\oplus B}(\alpha_{A}(a_1)+\alpha_{B}(b_1),\alpha_{A}(a_2)+\alpha_{B}(b_2))\\&=&
\mu_{A\oplus B}((\alpha_{A}+\alpha_{B})(a_1+b_1),(\alpha_{A}+\alpha_{B})(a_2+b_2)).
\end{array}$$
Similarly, we have $(\beta_{A}+\beta_{B})\circ\mu_{A\oplus B}(a_1+b_1,a_2+b_2)=\mu_{A\oplus B}((\beta_{A}+\beta_{B})(a_1+b_1),(\beta_{A}+\beta_{B})(a_2+b_2))$.

Secondly, we prove the left BiHom-alternative identity for $A\oplus B$ as follows
$$\begin{array}{llllll}
 &&as_{\alpha,\beta}((\beta_{A}+\beta_{B})(a_1+b_1),(\alpha_{A}+\alpha_{B})(a_1+b_1), a_2+b_2)\nonumber\\
 &=&\mu_{A\oplus B}(\mu_{A\oplus B}((\beta_{A}+\beta_{B})(a_1+b_1),(\alpha_{A}+\alpha_{B})(a_1+b_1)),(\beta_{A}+\beta_{B})(a_2+b_2))\nonumber\\
&&-\mu_{A\oplus B}((\alpha_{A}+\alpha_{B})(\beta_{A}+\beta_{B})(a_1+b_1), \mu_{A\oplus B}((\alpha_{A}+\alpha_{B})(a_1+b_1),a_2+b_2))\nonumber\\
 &=&\mu_{A\oplus B}(\mu_{A\oplus B}(\beta_{A}(a_1)+\beta_{B}(b_1),\alpha_{A}(a_1)+\alpha_{B}(b_1)),\beta_{A}(a_2)+\beta_{B}(b_2))\nonumber\\
&&-\mu_{A\oplus B}(\alpha_{A}\beta_{A}(a_1)+\beta_{B}\beta_{B}(b_1), \mu_{A\oplus B}(\alpha_{A}(a_1)+\alpha_{B}(b_1),a_2+b_2))\nonumber\\
&=&\mu_{A\oplus B}(\mu_A(\beta_{A}(a_1)+\alpha_{A}(a_1))+\mu_B(\beta_{B}(b_1),\alpha_{B}(b_1)),(\beta_{A}(a_2)+\beta_{B}(b_2))
\nonumber\\
&&-\mu_{A\oplus B}(\alpha_{A}\beta_{A}(a_1)+\alpha_{B}\beta_{B}(b_1),\mu_A(\alpha_{A}(a_1),a_2)+\mu_{B}(\alpha_{B}(b_1),b_2))
\nonumber\\
&=&\mu_A(\mu_A(\beta_{A}(a_1),\alpha_{A}(a_1)),\beta_{A}(a_2))+\mu_B(\mu_B(\beta_{B}(b_1),\alpha_{B}(b_1)),\beta_{B}(b_2))\nonumber\\
&&-\mu_A(\alpha_{A}\beta_{A}(a_1),\mu_A(\alpha_{A}(a_1),a_2))-\mu_B(\alpha_{B}\beta_{B}(b_1),\mu_B(\alpha_{B}(b_1),b_2))
\nonumber\\
&=&as_{\alpha_A,\beta_A}(\beta_{A}(a_1),\alpha_{A}(a_1),a_2)+as_{\alpha_B,\beta_B}(\beta_{B}(b_1),\alpha_{B}(b_1),b_2)=0+0=0 ~~(\textsl{by ~(\ref{is1})~ in}~ A~and~B).\nonumber
\end{array}$$
Similarly, we prove the right BiHom-alternative identity as follows:
$$\begin{array}{llll}
 &&as_{\alpha,\beta}(a_1+b_1,(\beta_{A}+\beta_{B})(a_2+b_2), (\alpha_{A}+\alpha_{B})(a_2+b_2))\\
 &=&\mu_{A\oplus B}(\mu_{A\oplus B}(a_{1}+b_{1},(\beta_{A}+\beta_{B})(a_2+b_2)),(\alpha_{A}+\alpha_{B})(\beta_{A}+\beta_{B})(a_2+b_2))\nonumber\\
&&-\mu_{A\oplus B}((\alpha_{A}+\alpha_{B})(a_1+b_1), \mu_{A\oplus B}((\beta_{A}+\beta_{B})(a_2+b_2),(\alpha_{A}+\alpha_{B})(a_2+b_2)))\nonumber\\
 &=&\mu_{A\oplus B}(\mu_{A\oplus B}(a_1+b_1,\beta_{A}(a_2)+\beta_{B}(b_2))),\alpha_{A}\beta_{A}(a_2)+\alpha_{B}\beta_{B}(b_2))\nonumber\\
&&-\mu_{A\oplus B}(\alpha_{A}(a_1)+\alpha_{B}(b_1), \mu_{A\oplus B}(\beta_{A}(a_2)+\beta_{B}(b_2),\alpha_{A}(a_2)+\alpha_{B}(b_2)))\nonumber\\
&=&\mu_{A\oplus B}(\mu_A(a_1,\beta_{A}(a_2))+\mu_B(b_1,\beta_{B}(b_2)),\alpha_{A}\beta_{A}(a_2)+\alpha_{B}\beta_{B}(b_2))
\nonumber\\
&&-\mu_{A\oplus B}(\alpha_{A}(a_1)+\alpha_{B}(b_1),\mu_A(\beta_{A}(a_1),\alpha_{A}(a_2))+\mu_{B}(\beta_{B}(b_2),\alpha_{B}(b_2)))
\nonumber\\
&=&\mu_A(\mu_A(a_1,\beta_{A}(a_2)),\alpha_{A}\beta_{A}(a_2))+\mu_B(\mu_B(b_1),\beta_{B}(b_2)),\alpha_{B}\beta_{B}(b_2))\nonumber\\
&&-\mu_A(\alpha_{A}(a_1),\mu_A(\beta_{A}(a_2),\alpha_{A}(a_2)))-\mu_B(\alpha_{B}(b_1),\mu_B(\beta_{B}(b_2),\alpha_{B}(b_2)))
\nonumber\\
&=&as_{\alpha_A,\beta_A}(a_{1},\beta_{A}(a_2),\alpha_{A}(a_2))+as_{\alpha_B,\beta_B}(b_1,\beta_{B}(b_2),\alpha_{B}(b_2))=0+0=0~~(\textsl{by ~(\ref{is2})~ in}~ A~and~B).\nonumber
\end{array}$$

Hence $(A\oplus B, \mu_{A\oplus B}, \alpha_{A}+\alpha_{B},\beta_{A}+\beta_{B})$ is a BiHom-alternative algebra.
\end{proof}
\begin{ex} \label{e1}
Consider the 2-dimensional of two BiHom-alternative algebras $(A,\alpha_{1},\beta_{1},\mu_{1})$ and $(A,\alpha_{2},\beta_{2},\mu_{2})$ with a basis $\{e_1,e_2\},$ (see \cite{GRAZIANI}) for $i=1,2$ the maps $\alpha_i$, $\beta_i $ and the multiplication $\mu_i$ are defined by
$$\begin{array}{llllllll}
  \alpha_1(e_1)=e_1, && \alpha_1(e_2)=\frac{2 a}{b-1}e_1-e_2,\\
  \beta_1(e_1)=e_1, && \beta_1(e_2)=-ae_1+b e_2, \\
  \mu_1(e_1,e_1)=e_1, && \mu_1(e_1,e_2)= -ae_1+b e_2, \\
  \mu_1(e_2,e_1)=\frac{2 a}{b-1}e_1-e_2, && \mu_1(e_2,e_2)=-\frac{
a^2(b-2)}{(b-1)^2}e_1+ae_2,
\end{array}$$
and
$$\begin{array}{llllll}
  \alpha_2(e_1)=e_1,&&\alpha_2(e_2)=\frac{b(1-a)}{a}e_1+a e_2, \\
 \beta_2(e_1)=e_1,&& \beta_2(e_2)=be_1+(1-a) e_2, \\
  \mu_2(e_1,e_1)=e_1, && \mu_2(e_1,e_2)= be_1+(1-a) e_2, \\
 \mu_2(e_2,e_1)=\frac{b(1-a)}{a}e_1+ae_2, && \mu_2(e_2,e_2)=\frac{b}{a}e_2,
\end{array}$$
where $a$, $b$ are parameters in $\mathbb{K}$, with $b\neq 1$ in the f\/irst case
and $a\neq 0$ in the second. Using the Proposition \ref{bylv}, the quadruple $(A\oplus A,\mu=\mu_{1}+\mu_{2},\alpha=\alpha_{1}+\alpha_{2},\beta=\beta_{1}+\beta_{2})$ is a BiHom-alternative algebra where, $$\begin{array}{llllllll}
  \alpha(e_1)=2e_1,&& \alpha(e_2)=\frac{2a^{2}+b(b-1)(1-a)}{a(b-1)}e_1+(a-1) e_2, \\
 \beta(e_1)=2e_1,&& \beta(e_2)=-(a+b)e_1+(b+1-a) e_2, \\
 \mu(e_1,e_1)=2e_1, && \mu(e_1,e_2)= (b-a)e_1+(b+1-a) e_2, \\
 \mu(e_2,e_1)=\frac{2a^{2}+b(b-1)(1-a)}{a(b-1)}e_1+(a-1) e_2, && \mu(e_2,e_2)=-\frac{
a^2(b-2)}{(b-1)^2}e_1+\frac{a^{2}+b}{a}e_{2},
\end{array}$$
\end{ex}
\begin{prop}
Let $(A,\mu_A,\alpha_{1},\beta_{1})$ and $(B,\mu_B,\alpha_{2},\beta_{2})$  be two BiHom-alternative algebras and $\varphi: A\rightarrow B$ be a linear map. Then
$\varphi$ is a morphism from the BiHom-alternative $(A,\mu_A,\alpha_{1},\beta_{1})$ to the BiHom-alternative algebra $(B,\mu_B,\alpha_{2},\beta_{2})$  if and only if its graph
$\Gamma_{\varphi}$ is a BiHom-subalgebra of $(A\oplus B, \mu_{A\oplus B}, \alpha_{1}+\beta_{1},\alpha_{2}+\beta_{2}).$
\end{prop}
\begin{proof}
  Let $\varphi: (A, \mu_A, \alpha_{1},\beta_{1})\longrightarrow (B, \mu_B ,\alpha_{2},\beta_{2})$ be a morphism of BiHom-alternative algebras.
Then for all $u, v\in A,$
$$ \mu_{A\oplus B}((u+\varphi(u),v+\varphi(v))=(\mu_A(u,v)+\mu_B(\varphi(u),\varphi(v)))=(\mu_A(u,v)+\varphi(\mu_A(u,v))).$$
Thus the graph $\Gamma_{\varphi}$ is closed under the multiplication $\mu_{A\oplus B}.$
Furthermore since $\varphi\circ\alpha_{1}=\alpha_{2}\circ\varphi$, we have $(\alpha_{1}\oplus\alpha_{2})(u, \varphi(u)) = (\alpha_{1}(u),
\alpha_{2}\circ\varphi(u)) = (\alpha_{1}(u), \varphi\circ\alpha_{1}(u)).$ In the same way, we have $(\beta_{1}\oplus\beta_{2})(u, \varphi(u)) = (\beta_{1}(u),
\beta_{2}\circ\varphi(u)) = (\beta_{1}(u), \varphi\circ\beta_{1}(u)),$
which implies that $\Gamma_{\varphi}$ is closed under $\alpha_{1}\oplus\alpha_{2}$ and $\beta_{1}\oplus\beta_{2}$ Thus $\Gamma_{\varphi}$ is a BiHom-subalgebra of
$(A\oplus B, \mu_{A\oplus B} , \alpha_{1}\oplus \alpha_{2},\beta_{1}+\beta_{2}).$\\
Conversely, if the graph $\Gamma_{\varphi}\subset A\oplus B$ is a
BiHom-subalgebra of
$(A\oplus B, \mu_{A\oplus B} , \alpha_{1}\oplus \alpha_{2},\beta_{1}+\beta_{2})$ then we have
$$\mu_{A\oplus B}((u+ \varphi(u)), (v+ \varphi(v)))=(\mu_A(u, v) + \mu_B(\varphi(u), \varphi(v)) )\in\Gamma_{\varphi},$$
which implies that
$$\mu_B(\varphi(u), \varphi(v))=\varphi(\mu_A(u, v)).$$
Furthermore, $(\alpha_{1}\oplus\alpha_{2})(\Gamma_{\varphi})\subset\Gamma_{\varphi},~(\beta_{1}\oplus\beta_{2})(\Gamma_{\varphi})\subset\Gamma_{\varphi}$ implies
$$(\alpha_{1}\oplus\alpha_{2})(u, \varphi(u))=(\alpha_{1}(u), \alpha_{2}\circ\varphi(u)) \in\Gamma_{\varphi},$$
$$(\beta_{1}\oplus\beta_{2})(u, \varphi(u))=(\beta_{1}(u), \beta_{2}\circ\varphi(u)) \in\Gamma_{\varphi}.$$
which is equivalent to the condition $\beta\circ\varphi(u)=\varphi\circ\alpha(u),$ i.e. $ \beta\circ\varphi=\varphi\circ\alpha.$ Therefore, $\varphi$ is a
morphism of BiHom-alternative algebras.
\end{proof}

\begin{prop}\label{prop0}
Let $(A,\mu,\alpha,\beta)$ be a BiHom alternative algebra, $\alpha',\beta':A\longrightarrow A$ be two algebra homomorphisms
such that any two of the maps $\alpha,\beta,\alpha',\beta'$ commute. Then $A_{\alpha',\beta'}=(A,\mu_{\alpha',\beta'}:=\mu\circ(\alpha'\otimes\beta'),\alpha\circ\alpha',\beta\circ\beta')$ is a
BiHom-alternative algebra. Moreover suppose that $(B, \mu', \gamma,\delta)$ is another BiHom-alternative algebra and $\gamma',\delta'$ be an
endomorphisms of $B$ suh that any two of the maps $\gamma,\delta,\gamma',\delta'$ commute.
If $f : (A,\mu, \alpha,\beta)\rightarrow (B, \mu', \gamma,\delta)$
is a morphism such that $f\circ\alpha'=\gamma\circ'f$ and $f\circ\beta'=\delta'\circ f$, then $f : A_{\alpha',\beta'}\rightarrow B_{\gamma',\delta'}$ is also a morphism.
\end{prop}
\begin{proof}
First it is easy to see that, $\forall x, y\in A$ we have

$\begin{array}{lllll}
&&(\alpha\circ\alpha')\circ\mu_{\alpha',\beta'}(x,y)=\mu_{\alpha',\beta'}(\alpha\circ\alpha'(x),\alpha\circ\alpha'(y)),\\
&&(\beta\circ\beta')\circ\mu_{\alpha',\beta'}(x,y)=\mu_{\alpha',\beta'}((\beta\circ\beta')(x),(\beta\circ\beta')(y)).
\end{array}$\\

Now, to prove the left and right BiHom-alternative identities for $A_{\alpha',\beta'}$ pick  $ x,y\in A.$ Then
$$\begin{array}{lll}
&&as_{\alpha\alpha',\beta\beta'}(\beta\beta'(x),\alpha\alpha'(x),y)\\&=&\mu_{\alpha',\beta'}(\mu_{\alpha',\beta'}(\beta\beta  '(x),\alpha\alpha'(x)),\beta\beta'(y))-\mu_{\alpha',\beta'}(\alpha\alpha'\beta\beta'(x),\mu_{\alpha',\beta'}(x),y))\\
&=&\mu_{\alpha',\beta'}(\mu(\alpha'\beta'\beta(x),\alpha\alpha'\beta'(x)),\beta\beta'(y))-\mu_{\alpha',\beta'}(\alpha\alpha'\beta\beta'(x),\mu(\alpha\alpha'^{2}(x),\beta'(y))\\
&=&\mu(\mu(\alpha'^{2}\beta'\beta(x),\alpha\alpha'\beta'(x)),\beta\beta'^{2}(y))-\mu(\alpha\alpha'^{2}\beta\beta'(x),\mu(\alpha\alpha'^{2}\beta'(x),\beta'^{2}(y)))\\
&=&as_{\alpha,\beta}(\beta(\alpha'^{2}\beta'(x)),\alpha(\alpha'^{2}\beta'(x)),\beta'^{2}(y))=0~~(\textsl{by ~(\ref{is1})~ in}~ A),
\end{array}$$
$$\begin{array}{lll}
&&as_{\alpha\alpha',\beta\beta'}(x,\beta\beta'(y),\alpha\alpha'(y))\\&=&\mu_{\alpha',\beta'}(\mu_{\alpha',\beta'}(x,\beta\beta'  (y)),\beta\alpha\alpha'(y))-\mu_{\alpha',\beta'}(\alpha\alpha'(x),\mu_{\alpha',\beta'}(\beta\beta'(y),\alpha\alpha'(y)))\\
&=&\mu_{\alpha',\beta'}(\mu(\alpha'(x),\beta\beta'^{2}(x)),\beta\alpha\alpha'(y))-\mu_{\alpha',\beta'}(\alpha\alpha'(x),\mu(\beta\alpha'\beta'(y),\alpha\alpha'\beta'(y)))\\
&=&\mu(\mu(\alpha'^{2}(x),\beta\beta'^{2}\alpha'(y)),\beta\beta'\alpha\alpha'(y))-\mu(\alpha\alpha'^{2}(x),\mu(\beta\alpha'\beta'^{2}(y),\alpha\alpha'\beta'^{2}(y)))\\
&=&as_{\alpha,\beta}(\alpha'^{2}(x),\beta(\alpha'\beta'^{2}(y)),\alpha(\alpha'\beta'^{2}(y)))=0~~(\textsl{by ~(\ref{is2})~ in}~A).
\end{array}$$
Then $A_{\alpha',\beta'}=(A,\mu_{\alpha',\beta'}:=\mu\circ(\alpha'\otimes\beta'),\alpha\circ\alpha',\beta\circ\beta')$ is a BiHom-alternative algebra.

Finally, we have:
$$\begin{array}{lll}
f\circ\mu_{\alpha',\beta'}(x,y)&=&f\circ\mu_{\alpha,\beta}(\alpha'(x),\beta'(y))\\
&=&\mu_{\gamma,\delta}\circ f(\alpha'(x),\beta'(y))\\
&=&\mu_{\gamma,\delta}(f\alpha'(x),f\beta'(y))\\
&=&\mu_{\gamma,\delta}(\gamma'f(x),\delta'f(y))\\
&=&\mu_{\gamma',\delta'}(f(x),f(y)).
\end{array}$$
This completes the proof.
\end{proof}

\begin{cor}\label{prop1}
Let $(A,\mu,\alpha,\beta)$ be a BiHom alternative algebra. Then $(A,\mu_{k}:=\mu\circ(\alpha^{k}\otimes\beta^{k}),\alpha^{k+1},\beta^{k+1})$ is a BiHom-alternative algebra.
\end{cor}
\begin{proof}
Apply Proposition \ref{prop0} with $\alpha'=\alpha^{k}$ and $\beta'=\beta^{k}$.
\end{proof}
\begin{prop}
Let $(\mathcal{A},\mu,\alpha_{1},\beta_{1})$ be a BiHom-associative algebra and $(A,\nu,\alpha_{2},\beta_{2})$ be a BiHom-alternative algebra.
Then $(\mathcal{A}\otimes A,\mu_{\mathcal{A}\otimes A},\alpha,\beta)$ is a BiHom-alternative algebra, where the bilinear map $\mu_{\mathcal{A}\otimes A}:(\mathcal{A}\otimes A){\times 2}\longrightarrow\mathcal{A}\otimes A$ are given by $$\mu_{\mathcal{A}\otimes A}(a_{1}\otimes x_{1},a_{2}\otimes x_{2})=\mu(a_{1},a_{2})\otimes\nu(x_{1},x_{2}),~\forall a_{i}\in\mathcal{A},~x_{i}\in A,~i=1,2,$$
and the two linear maps $\alpha,\beta:\mathcal{A}\otimes A\longrightarrow\mathcal{A}\otimes A$ are given by $\alpha(a_{1}\otimes x_{1})=\alpha_{1}(a_{1})\otimes \alpha_{2}(x_{1})$ and $\beta(a_{1}\otimes x_{1})=\beta_{1}(a_{1})\otimes \beta_{2}(x_{1})$.
\end{prop}
\begin{proof}
Since $\alpha_{1}\circ\beta_{1}=\beta_{1}\circ\alpha_{1}$ and $\alpha_{2}\circ\beta_{2}=\beta_{2}\circ\alpha_{2}$, we have $\alpha\circ\beta=\beta\circ\alpha$.

First, $(\alpha_{1}+\alpha_{2})$ is multiplicative with respect to $\mu_{\mathcal{A}\oplus A}.$ Indeed, $\forall (a_{1},x_{1}),~(a_{2},x_{2})\in \mathcal{A}\times A$
$$\begin{array}{llll}&&\alpha\circ\mu_{\mathcal{A}\otimes A}(a_{1}\otimes x_{1},a_{2}\otimes x_{2})\\&=&\alpha(\mu(a_{1},a_{2})\otimes\nu(x_{1},x_{2}))\\
&=&(\alpha_{1}\circ\mu(a_{1},a_{2})\otimes\alpha_{2}\nu(x_{1},x_{2}))\\
&=&\mu(\alpha_{1}(a_{1}),\alpha_{1}(a_{2}))\otimes\nu(\alpha_{2}(x_{1}),\alpha_{2}(x_{2}))\\
&=&\mu(\alpha_{1}(a_{1})\otimes \alpha_{2}(x_{1}),\alpha_{1}(a_{2}]\otimes\alpha_{2}(x_{2}))\\
&=&\mu_{\mathcal{A}\otimes A}(\alpha(a_{1}\otimes x_{1}),\alpha(a_{2}\otimes x_{2})).
\end{array}$$
Similarly, we have $\beta\circ\mu_{\mathcal{A}\otimes A}(a_{1}\otimes x_{1},a_{2}\otimes x_{2})=
\mu_{\mathcal{A}\otimes A}(\beta(a_{1}\otimes x_{1}),\beta(a_{2}\otimes x_{2}))$.

Now, for $a_{1},a_{2}\in\mathcal{A},~x_{1},x_{2}\in A,$ we prove the left and right BiHom-alternative identities for $\mathcal{A}\oplus A$ as follows
$$\begin{array}{llll}&&as_{\alpha,\beta}(\beta(a_{1}\otimes x_{1}),\alpha(a_{1}\otimes x_{1})),a_{2}\otimes x_{2})\\&=&\mu_{\mathcal{A}\otimes A}(\mu_{\mathcal{A}\otimes A}(\beta(a_{1}\otimes x_{1}),\alpha(a_{1}\otimes x_{1})),\beta(a_{2}\otimes x_{2}))\\&&-\mu_{\mathcal{A}\otimes A}(\alpha\beta(a_{1}\otimes x_{1}),\mu_{\mathcal{A}\otimes A}(\alpha(a_{1}\otimes x_{1})),a_{2}\otimes x_{2}))\\
&=&\mu_{\mathcal{A}\otimes A}(\mu_{\mathcal{A}\otimes A}(\beta_{1}(a_{1})\otimes\beta_{2}(x_{1}),\alpha_{1}(a_{1})\otimes\alpha_{2}(x_{1})),\beta_{1}(a_{2})\otimes \beta_{2}(x_{2}))\\&&-\mu_{\mathcal{A}\otimes A}(\alpha\beta(a_{1}\otimes x_{1}),\mu_{\mathcal{A}\otimes A}(\alpha_{1}(a_{1})\otimes\alpha_{2}(x_{1}),a_{2}\otimes x_{2}))\\
&=&\mu_{\mathcal{A}\otimes A}(\mu(\beta_{1}(a_{1}),\alpha_{1}(a_{1}))\otimes\nu(\beta_{2}(x_{1}),\alpha_{2}(x_{1})),\beta_{1}(a_{2})\otimes\beta_{2}(a_{2}))\\
&&-\mu_{\mathcal{A}\otimes A}(\alpha_{1}\beta_{1}(a_{1})\otimes\alpha_{2}\beta_{2}(x_{1}),\mu(\alpha_{1}(a_{1}),a_{2})\otimes\nu(\alpha_{2}(x_{1}),x_{2}))\\
&=&\mu(\mu(\beta_{1}(a_{1}),\alpha_{1}(a_{1})),\beta_{1}(a_{2}))\otimes\nu(\nu(\beta_{2}(x_{1}),\alpha_{2}(x_{1})),\beta_{2}(x_{2}))\\
&&-\mu(\alpha_{1}\beta_{1}(a_{1}),\mu(\alpha_{1}(a_{1}),a_{2}))\otimes\nu(\alpha_{2}\beta_{2}(x_{1}),\nu(\alpha_{2}(x_{1}),x_{2}))\\
&=&0~~(\textsl{by ~(\ref{BiHomAss})~in~$\mathcal{A}$~and~(\ref{is1}) in}~A),
\end{array}$$
$$\begin{array}{lll}
&&as_{\alpha,\beta}(a_{1}\otimes x_{1},\beta(a_{2}\otimes x_{2})),\alpha(a_{2}\otimes x_{2}))\\
&=&\mu_{\mathcal{A}\otimes A}(\mu_{\mathcal{A}\otimes A}(a_{1}\otimes x_{1},\beta(a_{2}\otimes x_{2})),\alpha\beta(a_{2}\otimes x_{2}))\\
&&-\mu_{\mathcal{A}\otimes A}(\alpha(a_{1}\otimes x_{1}),\mu_{\mathcal{A}\otimes A}(\beta(a_{2}\otimes x_{2}),\alpha(a_{2}\otimes x_{2})))\\
&=&\mu_{\mathcal{A}\otimes A}(\mu_{\mathcal{A}\otimes A}(a_{1}\otimes x_{1},\beta_{1}(a_{2})\otimes\beta_{2}(x_{2})),\alpha_{1}\beta_{1}(a_{2})\otimes\alpha_{2}\beta_{2}(x_{2}))\\
&&-\mu_{\mathcal{A}\otimes A}(\alpha_{1}(a_{1}\otimes\alpha_{2}(x_{1}),\mu_{\mathcal{A}\otimes A}(\beta_{1}(a_{2})\otimes\beta_{2}(x_{2}),\alpha_{1}(a_{2})\otimes\alpha_{2}(x_{2}))\\
&=&\mu_{\mathcal{A}\otimes A}(\mu(a_{1},\beta_{1}(a_{2}))\otimes\nu(x_{1},\beta_{2}(x_{2})),\alpha_{1}\beta_{1}(a_{2})\otimes\alpha_{2}\beta_{2}(x_{2}))\\
&&-\mu_{\mathcal{A}\otimes A}(\alpha_{1}a_{1})\otimes\alpha_{2}(x_{1}),\mu(\beta_{1}(a_{2}),\alpha_{1}(a_{2}))\otimes\nu(\beta_{2}(x_{2})\otimes\alpha_{2}(x_{2}))\\
&=&\mu(\mu(a_{1},\beta_{1}(a_{2}),\alpha_{1}\beta_{1}(a_{2}))\otimes\nu(\nu(x_{1}),\beta_{2}(x_{2})),\alpha_{2}\beta_{2}(x_{2}))\\
&&-\mu(\alpha_{1}(a_{1}),\mu(\beta_{1}(a_{2}),\alpha_{1}(a_{2}))\otimes\nu(\alpha_{2}(x_{1}),\nu(\beta_{2}(x_{2})\otimes\alpha_{2}(x_{2}))\\
&=&0~~(\textsl{by ~(\ref{BiHomAss})~in~$\mathcal{A}$~and~(\ref{is2}) in}~A).
\end{array}$$
Then $(\mathcal{A}\otimes A,\mu_{\mathcal{A}\otimes A},\alpha,\beta)$ is a BiHom-alternative algebra.
\end{proof}

\begin{thm} Let $(A, \mu, \a,\b, R)$ be a Rota-Baxter BiHom-alternative algebra of weight
$0$ such that $R$ commutes with $\a, \b.$ Define a new multiplication on $A$ by
\begin{eqnarray*}
\mu_{R}(x,y)=\mu(R(x), y)+\mu(x, R(y)),~~~\mbox{for any $x,y\in A$}.
\end{eqnarray*}
 Then $A_R=(A, \mu_{R}, \a,\b, R)$ is a  BiHom-alternative algebra.
\end{thm}
\begin{proof} The multiplication of $\a,\b$  with respect to $\mu_{R}$ follows from the one of $\a,\b$  with respect to $\mu$
  and the hypothesis $R\circ\a=\a\circ R,~ R\circ\b=\b\circ R$.
To prove the left and right  BiHom-alternative identities, let denote by
$as_{\alpha,\beta, \mu_R}$ the BiHom-assoiator in $A_R$ and $as_{\alpha,\beta}$
the one in the BiHom-algebra $(A, \mu, \a,\b)$. Then for all $x,y\in A,$
we have:
$$\begin{array}{llllll}
as_{\alpha,\beta,\mu_{R}}(\beta(x),\alpha(x),y)&=&\mu_{R}(\mu_{R}(\beta(x),\alpha(x)),\beta(y))-\mu_{R}(\alpha\beta(x),\mu_{R}(\alpha(x),y))\\
&=&\mu_{R}(\mu(\beta R(x),\alpha(x))+\mu(\beta(x),\alpha R(x)),\beta(y)))\\&&-\mu_{R}(\alpha\beta(x),\mu(\alpha R(x),y)+\mu(\alpha(x),R(y)))\\
&=&\mu(R(\mu(\beta R(x),\alpha(x))+\mu(\beta(x),\alpha R(x)),\beta(y))\\&&+\mu(\mu(\beta R(x),\alpha(x))+\mu(\beta(x),\alpha R(x)),\beta R(y)))\\
&&-\mu(\alpha\beta R(x),\mu(\alpha R(x),y)+\mu(\alpha(x),R(y)))\\&&-\mu(\alpha\beta(x),R(\mu(\alpha R(x),y)+\mu(\alpha(x),R(y))))\\
&=&\mu(\mu(\beta R(x),\alpha R(x)),\beta(y))+\mu(\mu(\beta R(x),\alpha(x)),\beta R(y))\\&&+\mu(\mu(\beta(x),\alpha R(x)),\beta R(y))
-\mu(\alpha\beta R(x),\mu(\alpha R(x),y))\\&&-\mu(\alpha\beta R(x),\mu(\alpha(x),R(y)))-\mu(\alpha\beta(x),\mu(\alpha (R(x)),R(y)))\\
&=&as_{\alpha,\beta}(\beta (R(x)),\alpha (R(x)),y)+as_{\alpha,\beta}(\beta (R(x)),\alpha (x),R(y))\\&&+as_{\alpha,\beta}(\beta (x),\alpha (R(x)),R(y))=0+0=0,
~~(\textsl{by ~(\ref{sa1})~and~(\ref{is1})}).
\end{array}$$
In the same way, the right BiHom-alternative identities $as_{\alpha,\beta,\mu_{R}}(x,\beta(y),\alpha(y))=0$.
Hence, $A_R=(A, \mu_{R}, \a,\b, R)$ is a  BiHom-alternative algebra.
\end{proof}
\subsection{Constructions of BiHom-Jordan algebras}
\begin{prop}
Let $(A,\mu,\alpha,\beta)$ be a BiHom-Jordan algebra and $I$ be a two-sided BiHom-ideal of $(A,\mu,\alpha,\beta).$ Then
$(A/I,\overline{\mu},\overline{\alpha},\overline{\beta})$ is a BiHom-Jordan algebra  where $\overline{\mu}(\overline{x},\overline{y})=\overline{\mu(x,y)},$
$\overline{\alpha}(\overline{x})=\overline{\alpha(x)}$ and $\overline{\beta}(\overline{x})=\overline{\beta(x)}$
for all $\overline{x},\overline{y}\in A/I.$
\end{prop}
\begin{proof} The multiplicativity  and the BiHom-Commutativity in $(A/I,\overline{\mu},\overline{\alpha},\overline{\beta})$ follow from those in
$(A,\mu,\alpha,\beta).$ Next, by a direct computation, we get for $\forall\ \overline{x},\overline{y}\in A/I:$
\begin{eqnarray}
&&as_{\overline{\alpha},\overline{\beta}}(\overline{\mu}(\overline{\beta}^2(\overline{x}),\overline{\alpha}\overline{\beta}(\overline{x})),
\overline{\alpha}^2\overline{\beta}(\overline{y}),\overline{\alpha}^3(\overline{x}))
=\overline{as_{\alpha,\beta}(\mu(\beta^2(x),\alpha\beta(x)),\alpha^2\beta(y),\alpha^3(x))}\nonumber
\end{eqnarray}
Thus the BiHom-Jordan identity in $(A/I,\overline{\mu},\overline{\alpha},\overline{\beta})$ follows from the one in $(A,\mu,\alpha,\beta).$
\end{proof}
 \begin{prop}
 Given two BiHom-Jordan algebras $(A,\mu_A,\alpha_A,\beta_A)$ and
 $(B,\mu_B,\alpha_B ,\beta_B),$ there is a BiHom-Jordan algebra
 $(A\oplus B, \mu, \alpha, \beta) $ where the bilinear map
 $\mu :(A\oplus B)^{\times 2}\longrightarrow A\oplus B$ is given by
 $$\mu(a_1+b_1,a_2+b_2)=\mu_A(a_1,a_2)+\mu_B(a_2,b_2), \mbox{  $\forall \  (a_1,a_2)\in A^{\times 2}$\ and \  $\forall\   (b_1,b_2)\in B^{\times 2}$}$$
 and the linear maps $\alpha,\beta \ : A\oplus B\longrightarrow A\oplus B$ are given by
 \begin{eqnarray}
 \alpha(a+b)&=& \alpha_A(a)+\alpha_B(b),\nonumber\\
 \beta(a+b)&=& \beta_A(a)+\beta_B(b),
 \mbox{  $\forall \  (a,b)\in A\times B$}.\nonumber
 \end{eqnarray}
 \end{prop}
\begin{proof}Note that the Multiplicativity and BiHom-commutativity in $(A\oplus B, \mu, \alpha, \beta)$ follows from the those in $(A,\mu_A,\alpha_A,\beta_A)$ and
 $(B,\mu_B,\alpha_B ,\beta_B)$. Next, by a direct computation, we get for all $(x_1,y_1)\in A^{\times 2}$ and $(x_2,y_2)\in B^{\times 2}$:
 \begin{eqnarray}
 &&as_{\alpha,\beta}(\mu(\beta^2(x_1+y_1),\alpha\beta(x_1+y_1)),\alpha^2\beta(x_2+y_2),\alpha^3(x_1+y_1))\nonumber\\
 &&=as_{\alpha_A,\beta_A}(\mu_A(\beta_A^2(x_1),\alpha_A\beta_A(x_1)),\alpha_A^2\beta_A(y_1),\alpha_A^3(x_1))\nonumber\\
 &&+as_{\alpha_B,\beta_B}(\mu_B(\beta_B^2(x_2),\alpha_B\beta_B(x_2)),\alpha_A^2\beta_A(y_2),\alpha_B^3(x_2))\nonumber
 \end{eqnarray}
 Thus the BiHom-Jordan identity in $(A\oplus B, \mu, \alpha, \beta)$ follows from the one in $(A,\mu_A,\alpha_A,\beta_A)$ and
 $(B,\mu_B,\alpha_B ,\beta_B).$\end{proof}
 \begin{prop}
 A linear map $\phi: (A,\mu_A,\alpha_A,\beta_A)\longrightarrow (B,\mu_B,\alpha_B ,\beta_B)$ is a morphism of BiHom-Jordan algebras if and only if its
  graph of $\Gamma_{\phi}\subseteq A\oplus B $ is a BiHom-subalgebra of
  $(A\oplus B, \mu=\mu_A+\mu_B, \alpha=\alpha_A+\alpha_B, \beta=\beta_A+\beta_B).$
 \end{prop}
\begin{proof}Suppose that  $\phi: (A,\mu_A,\alpha_A,\beta_A)\longrightarrow (B,\mu_B,\alpha_B ,\beta_B)$ is a morphism of BiHom-Jordan algebras. Then for all $(a, b)\in A^{\times 2}:$
 \begin{eqnarray}
 \mu(a+\phi(a),b+\phi(b))&=&\mu_A(a,b)+\mu_B(\phi(a),\phi(b))=
 \mu_A(a,b)+\phi\mu_A(a,b).\nonumber
 \end{eqnarray}
 Thus $\Gamma_{\phi}$ is closed under the bilinear map $\mu.$ Furthermore, since $\phi\alpha_A=\alpha_B\phi$ and $\phi\beta_A=\beta_B\phi$ we have
 \begin{eqnarray}
&& \alpha(a+\phi(a))=\alpha_A(a)+\alpha_B(\phi(a))=\alpha_A(a)+\phi\alpha_A(a)
\nonumber\\
&& \beta(a+\phi(a))=\beta_A(a)+\beta_B(\phi(a))=\beta_A(a)+\phi\beta_A(a)
\nonumber
 \end{eqnarray}
 which imply that $\Gamma$ is closed under $\alpha$ and $\beta.$ It follows that $\Gamma_{\phi} $ is a BiHom-subalgebra of\\
  $(A\oplus B, \mu, \alpha, \beta).$
  Conversely suppose that that $\Gamma_{\phi} $ is a BiHom-subalgebra of\\
  $(A\oplus B, \mu, \alpha, \beta).$ Then
  \begin{eqnarray}
  &&\mu(a+\phi(a),b+\phi(a))=\mu_A(a,b)+\mu_B(\phi(a),\phi(b))\in\Gamma_{\phi}
  \nonumber
  \end{eqnarray}
  which implies that $\mu_b(\phi(a),\phi(b))=\phi\mu_A(a,b).$ Furthermore
  $\alpha(\Gamma_{\phi})\subseteq \Gamma_{\phi}$ and
  $\beta(\Gamma_{\phi})\subseteq \Gamma_{\phi}$ imply
  \begin{eqnarray}
  \alpha(a+\phi(a))=\alpha_A(a)+\alpha_B(\phi(a))\in\Gamma_{\phi}
  \mbox{ and }\nonumber\\
  \beta(a+\phi(a))=\beta_A(a)+\beta_B(\phi(a))\in\Gamma_{\phi}
  \nonumber
  \end{eqnarray}
  which are equivalent to conditions $\alpha_B\phi(a)=\alpha_A\phi(a)$ and
  $\beta_B\phi(a)=\beta_A\phi(a)$ that is
  $\alpha_B\phi=\alpha_A\phi$ and
  $\beta_B\phi=\beta_A\phi.$ Therefore,
  $\phi: (A,\mu_A,\alpha_A,\beta_A)\longrightarrow (B,\mu_B,\alpha_B ,\beta_B)$ is a morphism of BiHom-Jordan algebras.\end{proof}
\begin{prop}\label{assoJord}
Let $(A,\mu, \alpha,\beta)$ be a regular BiHom-associative algebra. Then $A^+=(A,\mu',\alpha,\beta)$ is a BiHom-Jordan algebra where
$\mu'(x, y)= \mu(x, y)+\mu(\alpha^{-1}\beta(y),\beta^{-1}\alpha(x)),$ for all $x,y\in A.$
\end{prop}
\begin{proof}
It is easy to see that, $\forall x, y\in A$

$\begin{array}{lllll}
&&\alpha\circ\mu'(x,y)=\mu'(\alpha(x),\alpha(y))\\
&&\beta\circ\mu'(x,y)=\mu'(\beta(x),\beta(y)),
\end{array}$

and

$$\begin{array}{lllllll}
\mu'(\beta(x),\alpha(y))&=&\mu(\beta(x),\alpha(y))+\mu(\alpha^{-1}\beta(\alpha(y)),\alpha\beta^{-1}(\beta(x))\\
&=&\mu(\alpha^{-1}\beta(\alpha(x)),\beta^{-1}\alpha(\beta(y))+\mu(\beta(y),\alpha(x))~(\textsl{by~BiHom-commutativity ~of~} \mu)\\
&=&\mu'(\beta(y),\alpha(x))
\end{array}$$
Now, we prove the BiHom-Jordan identity, let denote the BiHom-assoiator in
$A^+$ by $as_{\alpha,\beta,\mu'}.$ Then for all $x,y\in A,$ we have:
$$\begin{array}{llll}
&&as_{\alpha,\beta,\mu'}(\mu'(\beta^{2}(x),\alpha\beta(x)),\alpha^{2}\beta(y),\alpha^{3}(x))\\
&=&\mu'(\mu'(\mu'(\beta^{2}(x),\alpha\beta(x)),\alpha^{2}\beta(y)),\alpha^{3}\beta(x))
-\mu'(\mu'(\alpha\beta^{2}(x),\alpha^{2}\beta(x)),\mu'(\alpha^{2}\beta(y),\alpha^{3}(x)))\\
&=&\mu'(\mu'(\mu(\beta^{2}(x),\alpha\beta(x))+\mu(\beta^{2}(x),\alpha\beta(x)),\alpha^{2}\beta(y)),\alpha^{2}\beta(x))\\
&&-\mu'(\mu(\alpha\beta^{2}(x),\alpha^{2}\beta(x))+\mu(\alpha\beta^{2}(x),\alpha^{2}\beta(x)),\mu(\alpha^{2}\beta(y),\alpha^{3}(x))+\mu(\alpha^{2}\beta(x)),\alpha^{3}(y)))\\
&=&2\mu'(\mu(\mu(\beta^{2}(x),\alpha\beta(x)),\alpha^{2}\beta(y))+\mu(\alpha\beta^{2}(y),\mu(\alpha\beta(x),\alpha^{2}(x))),\alpha^{3}\beta(x))\\
&&-2\mu'(\mu(\alpha\beta^{2}(x),\alpha^{2}\beta(x)),\mu(\alpha^{2}\beta(y),\alpha^{3}(x))-2\mu'(\mu(\alpha\beta^{2}(x),\alpha^{2}\beta(x)),\mu(\alpha^{2}\beta(x),\alpha^{3}(y))\\
&=&2\mu(\mu(\mu(\beta^{2}(x),\alpha\beta(x)),\alpha^{2}\beta(y)),\alpha^{3}\beta(x))
+2\mu(\alpha^{2}\beta^{2}(x),\mu(\mu(\alpha\beta(x),\alpha^{2}(x)),\alpha^{3}(y)))\\
&&+2\mu(\mu(\alpha\beta^{2}(y),\mu(\alpha\beta(x),\alpha^{2}(x))),\alpha^{3}\beta(x))+2\mu(\alpha^{2}\beta^{2}(x),\mu(\alpha^{2}\beta(y),\mu(\alpha^{2}(x),\beta^{-1}\alpha^{3}(x)))\\
&&-2\mu(\mu(\alpha\beta^{2}(x),\alpha^{2}\beta(x)),\mu(\alpha^{2}\beta(y),\alpha^{3}(x)))
-2\mu(\mu(\alpha\beta^{2}(y),\alpha^{2}\beta(x)),\mu(\alpha^{2}\beta(x),\alpha^{3}(x)))\\
&&-2\mu(\mu(\alpha\beta^{2}(x),\alpha^{2}\beta(x)),\mu(\alpha^{2}\beta(x),\alpha^{3}(y)))
-2\mu(\mu(\alpha\beta^{2}(x),\alpha^{2}\beta(y)),\mu(\alpha^{2}\beta(x),\alpha^{3}(x)))\\
&=&2\Big\{\mu(\mu(\mu(\beta^{2}(x),\alpha\beta(x)),\alpha^{2}\beta(y)),\alpha^{3}\beta(x))-\mu(\mu(\alpha\beta^{2}(x),\alpha^{2}\beta(x)),\mu(\alpha^{2}\beta(y),\alpha^{3}(x)))\Big\}\\
&&+2\Big\{\mu(\alpha^{2}\beta^{2}(x),\mu(\mu(\alpha\beta(x),\alpha^{2}(x)),\alpha^{3}(y)))-\mu(\mu(\alpha\beta^{2}(x),\alpha^{2}\beta(x)),\mu(\alpha^{2}\beta(x),\alpha^{3}(y)))\Big\}\\
&&+2\Big\{\mu(\mu(\alpha\beta^{2}(y),\mu(\alpha\beta(x),\alpha^{2}(x))),\alpha^{3}\beta(x))-\mu(\mu(\alpha\beta^{2}(y),\alpha^{2}\beta(x)),\mu(\alpha^{2}\beta(x),\alpha^{3}(x)))\Big\}\\
&&+2\Big\{\mu(\alpha^{2}\beta^{2}(x),\mu(\alpha^{2}\beta(y),\mu(\alpha^{2}(x),\beta^{-1}\alpha^{3}(x)))-\mu(\mu(\alpha\beta^{2}(x),\alpha^{2}\beta(y)),\mu(\alpha^{2}\beta(x),\alpha^{3}(x)))\Big\}\\
&=&0 ~~(\textsl{by~BiHom-associativity~condition~of}~\mu).
\end{array}$$
\end{proof}
\begin{ex} \label{e5}
Consider the 2-dimensional BiHom-associative algebras $(A,\mu,\alpha,\beta)$ with a basis $\{e_1,e_2\},$ the maps $\alpha$, $\beta $ and the multiplication $\mu$ are def\/ined by
$$\begin{array}{lllllll}
  \alpha(e_1)=e_1, && \alpha(e_2)=\frac{2 a}{b-1}e_1-e_2, \\
  \beta(e_1)=e_1, && \beta(e_2)=-ae_1+b e_2, \\
  \mu(e_1,e_1)=e_1, && \mu(e_1,e_2)= -ae_1+b e_2, \\
  \mu(e_2,e_1)=\frac{2 a}{b-1}e_1-e_2, && \mu(e_2,e_2)=-\frac{
a^2(b-2)}{(b-1)^2}e_1+ae_2,
\end{array}$$
where $a,b$ are parameters in $\mathbb{K}$, with $b\neq 0,1$.

It is clear that $(A,\mu,\alpha,\beta)$ is regular and $\alpha^{-1}(e_{1})=-e_{1},~\alpha^{-1}(e_{2})=\frac{2a}{b-1}e_{1}+e_{2},~\beta^{-1}(e_{1})=\frac{1}{b}e_{1}$ and $\beta^{-1}(e_{2})=\frac{a}{b}e_{1}+e_{2}$. Then using the product $\mu'$ in Proposition \ref{assoJord}, the quadruple $A^+=(A,\mu',\alpha,\beta)$ is a BiHom-Jordan algebra where
$$\begin{array}{llllllll}
\mu'(e_{1},e_{1})&=&\frac{b-1}{b}e_{1},&&
\mu'(e_{1},e_{2})&=&a(-1+\frac{1}{b}+\frac{4}{b-1})e_{1}+be_{2},\\
\mu'(e_{2},e_{1})&=&a(1+\frac{1}{b})e_{1}+(b-1)e_{2},&&
\mu'(e_{2},e_{2})&=&\Big(\frac{a^{2}(3b^{2}+11b-8)}{(b-1)^{2}}+\frac{a^{2}(1+b)}{b}\Big)e_{1}-\frac{4ab^{2}}{b-1}e_{2}.
\end{array}$$
\end{ex}
\section{BiHom-alternative algebra and BiHom-Jordan bimodules}
In this section, we give the defnition of BiHom-alternative and BiHom-Jordan (bi)modules. we construct several
different bimodules of BiHom-alternative BiHom-Jordan algebras. It is also proved that a direct sum of a BiHom-Jordan algebra and a bimodule over this BiHom-algebra is a BiHom-Jordan
algebra called a split null extension of the considered BiHom-algebra. Furthermore, relations between
BiHom-associative bimodules, BiHom-alternative and BiHom-Jordan bimodules and some examples are given.
\subsection{BiHom-alternative bimodules}
In this subsection, we study the BiHom-alternative bimodules.
First, we start by the following definitions.

\begin{defn}
Let $(A,\mu,\alpha,\beta)$ be a BiHom-alternative algebra.
\begin{enumerate}
\item A left BiHom-alternative $A$-module is a BiHom-module $(V,\phi,\psi)$ with a structure map $\rho_l:A\otimes V\longrightarrow V$, $a\otimes v\rightarrow a\cdot v$ such that
$$\begin{array}{lll}
as_{V_{\phi,\psi}}(\beta(x),\alpha(y),v)=-as_{V_{\phi,\psi}}(\beta(y),\alpha(x),v),~~~~\textrm{for} ~\textrm{all} ~x, y\in A~ \textrm{and}~ v\in V.
\end{array}$$
\item A right BiHom-alternative $A$-module is a BiHom-module $(V,\phi,\psi)$ with a structure map $\rho_r:V\otimes A\longrightarrow V,~v\otimes a\rightarrow v\cdot$ a  such that
$$\begin{array}{lll}
as_{V_{\phi,\psi}}(v,\beta(x),\alpha(y))=-as_{V_{\phi,\psi}}(v,\beta(y),\alpha(x)),~~~~\textrm{for} ~\textrm{all} ~x, y\in A~ \textrm{and}~ v\in V.
\end{array}$$
\end{enumerate}
\end{defn}
Now, as a generalization of  bimodules over BiHom-alternative algebras, one has:
\begin{defn}
Let $(A,\mu,\alpha,\beta)$ be a BiHom-alternative algebra.\\
A BiHom-alternative $A$-bimodule \cite{chtioui2} is a BiHom-module $(V,\phi,\psi)$  with a (left) structure map  $\rho_l:A\otimes V\longrightarrow V,$ $a\otimes v\mapsto a\cdot v$ and a (right) structure map $\rho_r:V\otimes A\longrightarrow V,$ $v\otimes a\mapsto  v\cdot a$
such that the following equalities hold:
\begin{eqnarray}
&& as_{V_{\phi,\psi}}(\beta(x),\alpha(x),v)=0\mbox{  $\forall(x,v)\in A\times V$ }\label{0}\\
&&as_{V_{\phi,\psi}}(v,\beta(x),\alpha(x))=0 \mbox{  $\forall(x,v)\in A\times V$ }\label{1}\\
&&as_{V_{\phi,\psi}}(\beta(x),\phi(v),y)=-as_{V_{\phi,\psi}}(\psi(v), \alpha(x),y)
\mbox{  $\forall(x,y,v)\in A^{\times 2}\times V$ } \label{2}\\
&& as_{V_{\phi,\psi}}(y,\beta(x),\phi(v))=-as_{V_{\phi,\psi}}(y,\psi(v),\alpha(x)) \mbox{  $\forall(x,y,v)\in A^{\times 2}\times V$ } \label{ra3}
\end{eqnarray}
\end{defn}
\begin{rmk}\label{rem2}
If $\alpha = \beta=Id_{A}$  and $\phi = \psi=Id_{V}$, then $V$ is the so-called alternative bimodule
for the alternative algebra $(A, \mu)$ \cite{Jacob2}.
\end{rmk}
\begin{ex}
There are some examples of BiHom-alternative $A$-bimodules.
\begin{enumerate}
\item
Let $(A,\mu,\alpha,\beta)$ be a BiHom-alternative algebra. Then $(A,\alpha,\beta)$ is a BiHom-alternative $A$-bimodule where the structure maps are
$\rho_{l}(a,b)=\mu(a,b)$ and $\rho_{r}(a,b)=\mu(b,a)$. More generally, if $B$ is a two-sided BiHom-ideal of $(A,\mu,\alpha,\beta)$, then $(B,\alpha,\beta)$ is a BiHom-alternative
$A$-bimodule where the structure maps are $\rho_{l}(a,x)=\mu(a,x)=\mu(x,a)=\rho_{r}(x,a)$ for all $x\in B$ and $(a,b)\in A^{\times 2}$.
\item
If $(A,\mu)$ is an alternative algebra and $V$ is an alternative $A$-bimodule \cite{Jacob2} in the usual sense, then $(V,Id_{V},Id_{V})$ is a BiHom-alternative $\mathbb{A}$-bimodule where $\mathbb{A}=(A,\mu,Id_{A},
Id_{A})$ is a BiHom-alternative algebra.\end{enumerate}
\end{ex}
\begin{prop}
If $f:(A,\mu_{A},\alpha_{A},\beta_{A})\longrightarrow(B,\mu_{B},\alpha_{B},\beta_{B})$ is a surjective morphism of BiHom-alternative algebras, then $(B,\alpha_{B},\beta_{B})$ becomes a
BiHom-alternative $A$-bimodule via $f$, i.e, the structure maps are defined as $\rho_{l}(a,b)=\mu_{B}(f(a),b)$ and $\rho_{r}(b,a)=\mu_{B}(b,f(a))$ for all
$(a,b)\in A\times B$.
\end{prop}
\begin{proof}First, we have
$$\begin{array}{lllll}
\alpha_{B}\rho_{l}&=&\alpha_{B}\mu_{B}(f\otimes Id_{B})\\
&=&\mu_{B}(\alpha_{B}\circ f\otimes\alpha_{B})\\
&=&\mu_{B}(f\circ\alpha_{A}\otimes \alpha_{B})\\
&=&\rho_{l}\circ(\alpha_{A}\otimes\alpha_{B})
\end{array}$$
 Similarly, we get that $\beta_{B}\rho_l=\rho_l\circ(\beta_{A}\otimes\beta_{B}),~\alpha_{B}\rho_r=\rho_r\circ(\alpha_{B}\otimes\alpha_{A})$ and $\beta_{B}\rho_r=\rho_r\circ(\beta_{B}\otimes\beta_{A})$.
Thus $\rho_l$ and $\rho_r$ are morphisms of BiHom-modules.

Next, for any elements $a\in A$ and $b\in B$, then there exist $b'\in A$ such that $f(b')=b$ (i.e. $f$ is surjective), and we have

$$\begin{array}{lllll}
&&as_{\alpha_{B},\beta_{B}}(\beta_{A}(a),\alpha_{A}(a),b)=as_{\alpha_{B},\beta_{B}}(\beta_{A}(a),\alpha_{A}(a),f(b'))\\
&=&\rho_{l}(\mu_{A}(\beta_{A}(a),\alpha_{A}(a)),\beta_{B} f(b'))-\rho_{l}(\alpha_{A}\beta_{A}(a),\rho_{l}(\alpha_{A}(a),f(b')))\\
&=&\mu_{B}(\mu_{B}(f\beta_{A}(a),f(\alpha_{A}(a))),\beta_{B} f(b'))-\mu_{B}(f(\alpha_{A}\beta_{A}(a)),\mu(f(\alpha_{A}(a))),f(b'))\\
&=&\mu_{B}(f(\mu_{A}(\beta_{A}(a),\alpha_{A}(a))),f(\beta_{A}(b')))-\mu_{B}(f(\alpha_{A}\beta_{A}(a)),f(\mu_{A}(\alpha_{A}(a),b')))\\
&=&f(\mu_{A}(\mu_{A}(\beta_{A}(a),\alpha_{A}(a)),\beta_{A}(b')))-f(\mu_{A}(\alpha_{A}\beta_{A}(a),\mu_{A}(\alpha_{A}(a),b')))\\
&=&f\circ as_{\alpha_{A},\beta_{A}}(\beta_{A}(a),\alpha_{A}(a),b')=0~~(by~(\textsl{\ref{is1})})
\end{array}$$
$$\begin{array}{lllll}&&as_{\alpha_{B},\beta_{B}}(b,\beta_{A}(a),\alpha_{A}(a))=as_{\alpha_{B},\beta_{B}}(f(b'),\beta_{A}(a),\alpha_{A}(a))\\
&=&\rho_{r}(\rho_{r}(f(b'),\beta_{A}(a)),\alpha_{A}\beta_{A}(a))-\rho_{r}(\alpha_{B} f(b'),\mu(\beta_{A}(a),\alpha_{A}(a)))\\
&=&\mu_{B}(\mu_{B}(f(b'),f(\beta_{A}(a))),f(\alpha_{A}\beta_{A}(a)))-\mu_{B}(\alpha_{B} f(b'),f(\mu_{A}(\beta_{A}(a),\alpha_{A}(a))))\\
&=&\mu_{B}(f(\mu_{A}(b',\beta_{A}(a))),f(\alpha_{A}\beta_{A}(a)))-\mu_{B}(f(\alpha_{A}(b')),f(\mu_{A}(\beta_{A}(a),\alpha_{A}(a))))\\
&=&f(\mu_{A}(\mu_{A}(b',\beta_{A}(a)),\alpha_{A}\beta_{A}(a)))-f(\mu_{A}(\alpha_{A}(b'),\mu_{A}(\beta_{A}(a),\alpha_{A}(a))))\\
&=&f\circ as_{\alpha_{A},\beta_{A}}(b',\beta_{A}(a),\alpha_{A}(a))=0~~(by~(\textsl{\ref{is2})})
\end{array}$$
Therefore, we get the (\ref{0}) and (\ref{1}). Finally, for all $a_{1},a_{2}\in A$ and $b=f(b')\in B$
$$\begin{array}{lllll}
&&as_{\alpha_{B},\beta_{B}}(\beta_{A}(a_{1}),\alpha_{B}(b),a_{2})=as_{\alpha_{B},\beta_{B}}(\beta_{A}(a_{1}),\alpha_{B}(f(b')),a_{2})\\
&=&\rho_{r}(\rho_{l}(\beta_{A}(a_{1}),\alpha_{B}(f(b'))),\beta_{A}(a_{2}))-\rho_{r}(\rho_{l}(\alpha_{B}(f(b')),a_{2}))\\
&=&\mu_{B}(\mu_{B}(f(\beta_{A}(a_{1})),\alpha_{B} f(b')),f(\beta_{A}(a_{2}))-\mu_{B}(f(\alpha_{A}\beta_{A}(a_{1}),\mu_{B}(\alpha_{B} f(b'),f(a_{2})))\\
&=&\mu_{B}(\mu_{B}(f(\beta_{A}(a_{1})),f\alpha_{A}(b')),f(\beta_{A}(a_{2})))-\mu_{B}(f(\alpha_{A}\beta_{A}(a_{1}),f(\mu_{A}(\alpha_{A}(b'),f(a_{2})))\\
&=&f\circ as_{\alpha_{A},\beta_{A}}(\beta_{A}(a_{1}),\alpha_{A}(b'),a_{2})\\
&=&-f\circ as_{\alpha_{A},\beta_{A}}(\beta_{A}(b'),\alpha_{A}(a_{1}),a_{2})~~(by~(\textsl{\ref{sa1})})\\
&=&-f(\mu_{A}(\mu_{A}(\beta_{A}(b'),\alpha_{A}(a_{1})),\beta_{A}(a_{2})))+f(\mu_{A}(\alpha_{A}\beta_{A}(b'),\mu_{A}(\alpha_{A}(a_{1}),a_{2}))\\
&=&-\rho_{r}(\rho_{r}(\beta_{B} f(b'),f(\alpha_{A}(a_{1}))),f(\beta_{A}(a_{2})))+\rho_{r}(\alpha_{B}\beta_{B} f(b'),\mu_{B}(f(\alpha_{A}(a_{1})),f(a_{2})))\\&=&-\rho_{r}(\rho_{r}(\beta_{B}(b),\alpha_{A}(a_{1})),\beta_{A}(a_{2}))+\rho_{r}(\alpha_{B}\beta_{B}(b),\mu_{A}(\alpha_{A}(a_{1}),a_{2}))\\
&=&-as_{\alpha_{B},\beta_{B}}(\beta_{B}(b),\alpha_{A}(a_{1}),a_{2}),
\end{array}$$
then the relation (\ref{2}) holds.
Similarly, we have  $$as_{\alpha_{B},\beta_{B}}(a_{2},\beta_{A}(a_{1}),\alpha_{B}(b))=-as_{\alpha_{B},\beta_{B}}(a_{2},\beta_{B}(b),\alpha_{A}(a_{1})).$$

Hence, the conclusion holds.

\end{proof}
\begin{defn}
An abelian extension of BiHom-alternative algebras is a short exact sequence of BiHom-alternative algebras
$$0\longrightarrow (V,\phi,\psi)\stackrel{\mbox{i}} \longrightarrow(A,\mu_{A},\alpha_{A},\beta_{A})\stackrel{\mbox{$\pi$}}\longrightarrow (B,\mu_{B},\alpha_{B},\beta_{B})\longrightarrow 0 ,$$
where $(V,\phi,\psi)$ is a trivial BiHom-alternative algebra, $i$ and $\pi$ are morphisms of BiHom-algebras. Furthermore, if there exists a morphism $s:(B,\mu_{B},\alpha_{B},\beta_{B})
\longrightarrow (A,\mu_{A},\alpha_{A},\beta_{A})$ such that $\pi\circ s=id_{B}$ then the abelian extension is said to be split and $s$ is called a section of $\pi$.
\end{defn}
\begin{ex}
Given an abelian extension as in the previous definition,  the BiHom-module
$(V,\alpha_V, \beta_V)$ inherits a structure of a BiHom-alternative $B$-bimodule and the actions of the BiHom-algebra $(B,\mu_B,\alpha_B,\beta_B)$
on $V$ are as follows. For any $x\in B,$ there exist $\tilde{x}\in A$ such that $x=\pi(\tilde{x}).$ Let $x$ acts on $v\in V$ by
$x\cdot v:=\mu_A(\tilde{x},i(v))$ and $v\cdot x:=\mu_A(i(v),\tilde{x}).$
These are well-defined, as another lift $\tilde{x'}$
of $x$ is written $\tilde{x'}=\tilde{x}+v'$ for some $v'\in V$ and thus
$x\cdot v=\mu_A(\tilde{x},i(v))=\mu_A(\tilde{x'},i(v))$ and $v\cdot x=\mu_A(i(v),\tilde{x})=\mu_A(i(v),\tilde{x'})$ because $V$ is trivial.
The actions property follow from the BiHom-alternativity identity.
In case these actions of $B$ on $V$  are trivial, one speaks of a central extension.
\end{ex}
The following results describes a sequence of bimodules over a BiHom-alternative algebra by twisting
the structure maps of a given bimodule over this BiHom-alternative algebra.
\begin{prop}\label{sma}
Let $(A,\mu,\alpha,\beta)$ be a BiHom-alternative algebra and  $V_{\phi,\psi}=(V,\phi,\psi)$  be a BiHom-alternative $A$-bimodule with the structure maps $\rho_l$ and $\rho_r$. Then the maps
\begin{eqnarray}
\rho_l^{(n,m)}=\rho_l\circ(\alpha^n\beta^{m}\otimes Id_V)\label{nma1}\\
\rho_r^{(n,m)}=\rho_r\circ(Id_V\otimes \alpha^n\beta^{m})\label{nma2}
\end{eqnarray}
give the BiHom-module $(V,\phi,\psi)$ the structure of a BiHom-alternative $A$-bimodule that we denote
by $V_{\phi,\psi}^{(n,m)}$ for each $n,m\in\mathbb{N}$.
\end{prop}
\begin{proof}
Since $\rho_{l}$ and $\rho_{r}$ are structure maps on $V$, it is clear  that $\rho_l^{(n,m)}$  and $\rho_r^{(n,m)}$ are structure maps. Next, observe that for all $x\in A$ and $v\in V,$
$$\begin{array}{llll}
&&as_{V_{\phi,\psi}^{(n,m)}}(\beta(x),\alpha(x),v)\\&=&\rho_{l}^{(n,m)}(\mu(\beta(x),\alpha(x)),\psi(v))-\rho_l^{(n,m)}(\alpha\beta(x)
,\rho_{l}^{(n,m)}(\alpha(x),v))\nonumber\\
&=&\rho_{l}(\mu(\alpha^{n}\beta^{m+1}(x),\alpha^{n+1}\beta^{m}(x)),\psi(v))-\rho_l^{(n,m)}(\alpha^{n+1}\beta^{m+1}(x),\rho_{l}(\alpha^{n+1}\beta^{m}(x),v))\nonumber\\
&=&as_{V_{\phi,\psi}}(\beta(\alpha^n\beta^{m}(x)),\alpha(\alpha^n\beta^{m}(x)),v)\\
&=&0~~(\textsl{by~(\ref{0})~in ~$V_{\phi,\psi}$})
\end{array}$$
$$\begin{array}{llll}
&&as_{V_{\phi,\psi}^{(n,m)}}(v,\beta(x),\alpha(x))\\
&=&\rho_{r}^{(n,m)}(\rho_{r}^{(n,m)}(v,\beta(x)),\alpha\beta(x))-\rho_{r}^{(n,m)}(\alpha(v),\mu(\beta(x),\alpha(x))\\
&=&\rho_{r}(\rho_{r}(v,\beta(\alpha^{n}\beta^{m}(x))),\alpha\beta(\alpha^{n}\beta^{m}(x)))-\rho_{r}(\alpha(v),\mu(\beta(\alpha^{n}\beta^{m}(x)),\alpha(\alpha^{n}\beta^{m}(x)))\\
&=&as_{V_{\phi,\psi}}(v,\beta(\alpha^{n}\beta^{m}(x)),\alpha(\alpha^{n}\beta^{m}(x)))\\
&=&0,~~(\textsl{by~(\ref{1})~in ~$V_{\phi,\psi}$}).
\end{array}$$
Therefore, we get (\ref{0}) and (\ref{1}) for $V^{(n,m)}$. Finally, for all $x,y\in A$ and $v\in V,$
$$\begin{array}{llll}
&&as_{V_{\phi,\psi}^{(n,m)}}(\beta(x),\phi(v),y)\\&=&\rho_r^{(n,m)}(\rho_l^{(n,m)}(\beta(x),\phi(v)),\beta(y))-\rho_l^{(n,m)}(\alpha(\beta(x)),
\rho_r^{(n,m)}(\phi(v),y))\nonumber\\
&=&\rho_r(\rho_l(\beta (\alpha^n\beta^{m}(x)),\phi(v)),\beta(\alpha^{n}\beta^{m}(y)))-\rho_l(\alpha\beta(\alpha^{n}\beta^{m}(x)),
\rho_r(\phi(v),(\alpha^n\beta^{m}(y)))\nonumber\\
&=&as_{V_{\phi,\psi}}(\beta(\alpha^n\beta^{m}(x)),\phi(v),\alpha^n\beta^{m}(y)))\nonumber\\
&=&-as_{V_{\phi,\psi}}(\psi(v),\alpha(\alpha^{n}\beta^{m}(x)),\alpha^n\beta^{m}(y))~~(\textsl{by~(\ref{2})~in ~$V_{\phi,\psi}$})\nonumber\\
&=&-\rho_r(\rho_r(\psi(v),\alpha(\alpha^{n}\beta^{m}(x))),\beta(\alpha^{n}\beta^{m}(y)))+\rho_r(\psi\phi(v),\mu(\alpha(\alpha^{n}\beta^{m}(x)),\alpha^{n}\beta^{m}(y)))\\
&=&-\rho_r^{(n,m)}(\rho_r^{(n,m)}(\psi(v),\alpha(x)),\beta(y))+\rho_r^{(n,m)}(\psi\phi(v),\mu(\alpha(x),y))\nonumber\\
&=&-as_{V_{\phi,\psi}^{(n,m)}}(\psi(v),\alpha(x),y).
\end{array}$$Then the relation (\ref{2}) holds in $V_{\phi,\psi}^{(n,m)}.$
Similarly, we prove that
$$as_{V_{\phi,\psi}^{(n,m)}}(y,\beta(x),\phi(v))=-as_{V_{\phi,\psi}^{(n,m)}}(y,\psi(v),\alpha(x)).$$
This completes the proof.
\end{proof}
\begin{thm}\label{mam}
Let $(A,\mu,\alpha,\beta)$ be an BiHom-alternative algebra, $V_{\phi,\psi}=(V,\phi,\psi)$ be a BiHom-alternative $A$-bimodule with the structure maps
$\rho_l,$ $\rho_r.$ Let $\alpha',\beta'$ be endomorphisms of the BiHom-alternative algebra $(A,\mu,\alpha,\beta)$ such that any two of the maps $\alpha,\alpha',\beta,\beta'$ commute
and $\phi',~\psi'$ be linear self-maps of $V$ such that any two of the maps $\phi,\phi',\psi,\psi'$ commute. Suppose furthermore that
$$\left\{
   \begin{array}{ll}
    \phi'\circ\rho_l=\rho_l\circ(\alpha'\otimes\phi')& \\
     \psi'\circ\rho_l=\rho_l\circ(\beta'\otimes\psi'),&
   \end{array}
 \right.
~~and~~
\left\{
   \begin{array}{ll}
     \phi'\circ\rho_r=\rho_r\circ(\phi'\otimes\alpha') &  \\
     \psi'\circ\rho_r=\rho_r\circ(\psi'\otimes\beta').&
   \end{array}
 \right.$$

and write $A_{\alpha',\beta'}$ for the BiHom-alternative algebra $(A,\mu_{\alpha',\beta'},\alpha\alpha',\beta\beta')$ and
$V_{\phi\phi',\psi\psi'}$ for the BiHom-module $(V,\phi\phi',\psi\psi').$ Then the maps:
\begin{eqnarray}
\tilde{\rho_l}=\rho_l\circ(\alpha'\otimes\psi') \mbox{ and }
\tilde{\rho_r}=\rho_r\circ(\phi'\otimes\beta')
\end{eqnarray}
give the BiHom-module $V_{\phi\phi',\psi\psi'}$ the structure of BiHom-alternative $A_{\alpha',\beta'}$-bimodule.
\end{thm}
\begin{proof}
Trivially, the maps $\tilde{\rho_l}$ and $\tilde{\rho_r}$ are morphisms of BiHom-modules.
Next, for all $x\in A$ and $v\in V,$ we have
$$\begin{array}{llll}
&&as_{V_{\phi\phi',\psi\psi'}}(\beta\beta'(x),\alpha\alpha'(x),v)\\
&=&\widetilde{\rho_{l}}(\mu_{\alpha',\beta'}(\beta\beta'(x),\alpha\alpha'(x)),\psi\psi'(v))-\widetilde{\rho_{l}}(\alpha\alpha'\beta\beta'(x),\widetilde{\rho_{l}}(\alpha\alpha'(x),v))\\
&=&\widetilde{\rho_{l}}(\mu(\beta\beta'\alpha'(x),\alpha\alpha'\beta'(x)),\psi\psi'(v))-\widetilde{\rho_{l}}(\alpha\alpha'\beta\beta'(x),\mu(\alpha\alpha'^{2}(x),\psi'(v))))\\
&=&\rho_{l}(\mu(\beta\beta'\alpha'^{2}(x),\alpha\alpha'^{2}\beta'(x)),\psi\psi'^{2}(v))-\rho_{l}(\alpha\alpha'^{2}\beta\beta'(x),\rho_{l}(\alpha\alpha'^{2}\beta'(x),\psi'^{2}(v)))\\
&=&as_{V_{\phi,\psi}}(\beta(\alpha'^{2}\beta'(x)),\alpha(\alpha'^{2}\beta'(x)),\psi'^{2}(v)))\\
&=&0,~~(\textsl{by~(\ref{0})~in ~$V_{\phi,\psi}$})
\end{array}$$
$$\begin{array}{lllllll}
&&as_{V_{\phi\phi',\psi\psi'}}(v,\beta\beta'(x),\alpha\alpha'(x))\\
&=&\widetilde{\rho_{r}}(\widetilde{\rho_{r}}(v,\beta\beta'(x)),\alpha\alpha'\beta\beta'(x))-\widetilde{\rho_{r}}(\phi\phi'(v),\mu_{\alpha',\beta'}(\beta\beta'(x),\alpha\alpha'(x)))\\
&=&\widetilde{\rho_{r}}(\rho_{r}(\phi'(v),\beta\beta'^{2}(x)),\alpha\alpha'\beta\beta'(x))-\widetilde{\rho_{r}}(\phi\phi'(v),\mu(\alpha'\beta\beta'(x),\alpha\alpha'\beta\beta'(x)))\\
&=&\rho_{r}(\rho_{r}(\phi'^{2}(v),\alpha'\beta\beta'^{2}(x)),\alpha\alpha'\beta\beta'^{2}(x))-\rho_{r}(\phi\phi'^{2}(v),\mu(\alpha'\beta\beta'^{2}(x),\alpha\alpha'\beta'^{2}(x)))\\
&=&as_{V_{\phi,\psi}}(\phi'^{2}(v),\beta(\alpha'\beta'^{2}(x)),\alpha(\alpha'\beta'^{2}(x)))\\
&=&0.~~(\textsl{by~(\ref{1})~in ~$V_{\phi,\psi}$})
\end{array}$$
Therefore, we get (\ref{0}) and (\ref{1}) for $V_{\phi\phi',\psi\psi'}$. Now, for all $x,y\in A$ and $v\in V,$
$$\begin{array}{llllll}
&&as_{V_{\phi\phi',\psi\psi'}}(\beta\beta'(x),\phi\phi'(v),y)\\
&=&\widetilde{\rho_{r}}(\widetilde{\rho_{l}}(\beta\beta'(x),\phi\phi'(v)),\beta\beta'(y))-\widetilde{\rho_{l}}(\alpha\alpha'\beta\beta'(x),\widetilde{\rho_{r}}(\phi\phi'(v),y)))\\
&=&\widetilde{\rho_{r}}(\rho_{l}(\alpha'\beta\beta'(x),\psi'\phi\phi'(v)),\beta\beta'(y))-\widetilde{\rho_{l}}(\alpha\alpha'\beta\beta'(x),\rho_{r}(\phi\phi'^{2}(v),\beta'(y)))\\
&=&\rho_{r}(\rho_{l}(\alpha'^{2}\beta\beta'(x),\psi'\phi\phi'^{2}(v)),\beta\beta'^{2}(y)))-\rho_{l}(\alpha\alpha'^{2}\beta\beta'(x),\rho_{r}(\phi\phi'^{2}\psi'(v),\beta'^{2}(y)))\\
&=&as_{V_{\phi,\psi}}(\beta(\alpha'^{2}\beta'(x)),\phi(\phi'^{2}\psi'(v)),\beta'^{2}(y)))\\
&=&-as_{V_{\phi,\psi}}(\psi(\phi'^{2}\psi'(v)),\alpha(\alpha'^{2}\beta'(x)),\beta'^{2}(y)),~~(\textsl{by~(\ref{2})~in ~$V_{\phi,\psi}$})\\
&=&-\rho_{r}(\rho_{r}(\psi(\phi'^{2}\psi'(v),\alpha\alpha'^{2}\beta'(x)),\beta\beta'^{2}(y))+\rho_{r}(\phi\psi\phi'^{2}\psi'(v),\mu(\alpha\alpha'^{2}\beta'(x),\beta'^{2}(y))\\
&=&-\rho_{r}(\widetilde{\rho_{r}}(\psi\phi'\psi'(v)),\alpha\alpha'^{2}(x)),\beta\beta'^{2}(y))+\rho_{r}(\phi\psi\phi'^{2}\psi'(v),\mu_{\alpha',\beta'}(\alpha\alpha'\beta'(x),\beta'(y))\\
&=&-\widetilde{\rho_{r}}(\widetilde{\rho_{r}}(\psi\psi'(v),\alpha\alpha'(x)),\beta\beta'(y))+\widetilde{\rho_{r}}(\phi\psi\phi'\psi'(v),\mu_{\alpha',\beta'}(\alpha\alpha'(x),y)\\
&=&-as_{V_{\phi\phi',\psi\psi'}}(\psi\psi'(v),\alpha\alpha'(x),y).
\end{array}$$ Then the relation (\ref{2}) holds in $V_{\phi\phi',\psi\psi'}.$
Similarly, we get\\
$as_{V_{\phi\phi',\psi\psi'}}(y,\beta\beta'(x),\phi\phi'(v))=as_{V_{\phi\phi',\psi\psi'}}(y,\psi\psi'(v),\alpha\alpha'(x)).$

Hence, the conclusion holds.
\end{proof}
\begin{cor}\label{isy}
Let $(A,\mu,\alpha,\beta)$ be a BiHom-alternative algebra, $(V,\phi,\psi)$ be a BiHom-alternative $A$-bimodule with the structure maps
$\rho_l$ and $\rho_r.$ Let $\alpha',\beta'$ be endomorphisms of the BiHom-alternative algebra $(A,\mu,\alpha,\beta)$ such that any two of the maps $\alpha,\alpha',\beta,\beta'$ commute,
and $\phi',~\psi'$ be linear self-maps of $V$ such that any two of the maps $\phi,\phi',\psi,\psi'$ commute. Suppose furthermore that
$$\left\{
   \begin{array}{ll}
    \phi'\circ\rho_l=\rho_l\circ(\alpha'\otimes\phi')& \\
     \psi'\circ\rho_l=\rho_l\circ(\beta'\otimes\psi'),&
   \end{array}
 \right.
~~and~~
\left\{
   \begin{array}{ll}
     \phi'\circ\rho_r=\rho_r\circ(\phi'\otimes\alpha') &  \\
     \psi'\circ\rho_r=\rho_r\circ(\psi'\otimes\beta').&
   \end{array}
 \right.$$

Write $A_{\alpha',\beta'}$ for the BiHom-alternative algebra $(A,\mu_{\alpha',\beta'},\alpha\circ\alpha',\beta\circ\beta')$ and
$V_{\phi\phi',\psi\psi'}$ for the BiHom-module $(V,\phi\phi',\psi\psi').$ Then the maps:
\begin{eqnarray}
\tilde{\rho_l}^{(n,m)}=\rho_l\circ(\alpha^{n}\beta^{m}\alpha'\otimes\psi') \mbox{ and }
\tilde{\rho_r}^{(n,m)}=\rho_r\circ(\phi'\otimes\alpha^{n}\beta^{m}\beta')
\end{eqnarray}
give the BiHom-module $V_{\phi\phi',\psi\psi'}$ the structure of a BiHom-alternative $A_{\alpha',\beta'}$-bimodule for each $n,m\in\mathbb{N}$.
\end{cor}
\begin{proof}
The proof follows from Proposition \ref{sma} and Theorem \ref{mam}.
\end{proof}
\begin{cor}\label{ismail}
Let $(A,\mu,\alpha,\beta)$ be a BiHom-alternative algebra, $(V,\phi,\psi)$ be a BiHom-alternative $A$-bimodule with the structure maps
$\rho_l$ and $\rho_r$.
Then writting\\ $A_{\alpha^{r},\beta^{s}}=(A,\m_{\alpha^{r},\beta^{s}}=\mu(\alpha^r\otimes\beta^s),\alpha^r,\beta^s),$  the maps:
\begin{eqnarray}
\tilde{\rho_l}^{(n,m,q)}=\rho_l\circ(\alpha^{n+r}\beta^{m}\otimes\psi^{q}) \mbox{ and }
\tilde{\rho_r}^{(n,m,p)}=\rho_r\circ(\phi^{p}\otimes\alpha^{n}\beta^{m+s})
\end{eqnarray}
give the BiHom-module $V_{\phi^{p},\psi^{q}}=(V,\phi^p,\psi^q)$ the structure of a BiHom-alternative $A_{\alpha^{r},\beta^{s}}$-bimodule for each $n,m,p,q,r,s\in\mathbb{N}.$
\end{cor}
\begin{proof}
Apply Corollary \ref{isy} with $\alpha'=\alpha^{r},~\beta'=\beta^{s}$ and $\phi'=\phi^{p},~\psi'=\psi^{q}$.
\end{proof}
\begin{prop}
Let $(A,\mu,\alpha,\beta)$ be a BiHom-alternative algebra, $V_{\phi,\psi}=(V,\phi,\psi)$ be a BiHom-alternative $A$-bimodule with the structure maps
$\rho_l$ and $\rho_r$ and $R: A\rightarrow A$ a be Rota-Baxter
operator of weight $0$ on $A$. Write $A_{R}$ for the BiHom-alternative algebra $(A,\mu_{R},\alpha,\beta).$  Then the maps defined on $V\otimes A$ and $A\otimes V$ respectively  by
\begin{eqnarray}
\widetilde{\rho}_l(x,v)=\rho_{l}(R\otimes Id_{V}).\label{nma1}\\
\widetilde{\rho}_r(v,x)=\rho_{r}(Id_{V}\otimes R)\label{nma2}
\end{eqnarray}
give the BiHom-module $V_{\phi,\psi}$ the structure of a BiHom-alternative $A_{R}$-bimodule that we denote by $\widetilde{V}_{\phi,\psi}.$
\end{prop}
\begin{proof} Since the structure map $\rho_l$ is a morphism of BiHom-modules, we get:
$$\begin{array}{lllll}
\phi\tilde{\rho_l}&=&\phi\rho_l\circ(R\otimes Id_V)\\
&=&\rho_l\circ(\alpha\circ R\otimes\phi)\\ 
&=&\rho_l\circ(R\circ\alpha\otimes \phi)\\
&=&\widetilde{\rho_l}\circ(\alpha\otimes\phi)\nonumber
\end{array}$$
 Similarly, we get that $\psi\widetilde{\rho_l}=\widetilde{\rho_l}\circ(\beta\otimes\psi),\
 \phi\widetilde{\rho_r}=\widetilde{\rho_r}\circ(\phi\otimes\alpha)$ and $\psi\widetilde{\rho_r}=\widetilde{\rho_r}\circ(\psi\otimes\beta)$. Thus $\widetilde{\rho_l}$ and $\widetilde{\rho_r}$ are morphisms of BiHom-modules.
Next, for any elements $x, y\in A$ and $v\in V$, we have
  $$\begin{array}{lll}
&&as_{\widetilde{V}_{\phi,\psi}}(\beta(x),\alpha(x),v)\\
&=&\widetilde{\rho_{l}}(\mu_{R}(\beta(x),\alpha(x)),\psi(v))-\widetilde{\rho_{l}}(\alpha\beta(x),\widetilde{\rho_{l}}(\alpha(x),v)\\
&=&\rho_{l}(R(\mu(R(\beta(x),\alpha(x))+\mu(\beta(x),R(\alpha(x))),\psi(v))-\rho_{l}(R(\alpha\beta(x)),\rho_{l}(R(\alpha(x)),v))\\
&=&\rho_{l}(\mu(\beta(R(x)),\alpha(R(x)),\psi(v))-\rho_{l}(\alpha\beta(x),\rho_{l}(\alpha(R(x)),v))\\
&=&as_{V_{\phi,\psi}}(\beta(R(x)),\alpha(R(x)),v)=0,~~(\textsl{by~(\ref{0})~in ~$V_{\phi,\psi}$})
  \end{array}$$
 $$\begin{array}{lll}
&&as_{\widetilde{V}_{\phi,\psi}}(v,\beta(x),\alpha(x))\\
&=&\widetilde{\rho_{r}}(\widetilde{\rho_{r}}(v,\beta(x)),\alpha\beta(x))-\widetilde{\rho_{r}}(\phi(v),\mu_{R}(\beta(x),\alpha(x))\\
&=&\widetilde{\rho_{r}}(\rho_{r}(v,R(\beta(x))),\alpha\beta(x))-\widetilde{\rho_{r}}(\phi(v),\mu(R(\beta(x)),\alpha(x)+\mu(\beta(x),R(\alpha(x)))\\
&=&\rho_{r}(\rho_{r}(v,R(\beta(x))),R(\alpha\beta(x)))-\rho_{r}(\phi(v),R(\mu(R(\beta(x)),\alpha(x))+\mu(\beta(x),R(\alpha(x))))\\
&=&\rho_{r}(\rho_{r}(v,\beta(R(x))),\alpha\beta(R(x)))-\rho_{r}(\phi(v),(\mu(\beta(R(x)),\alpha(R(x)))\\
&=&as_{V_{\phi,\psi}}(v,\beta(R(x)),\alpha(R(x)))=0,~~(\textsl{by~(\ref{1})~in ~$V_{\phi,\psi}$})
  \end{array}$$
  $$\begin{array}{lll}
&&as_{\widetilde{V}_{\phi,\psi}}(\beta(x),\phi(v),y)\\
&=&\widetilde{\rho_{r}}(\widetilde{\rho_{l}}(\beta(x),\phi(v)),\beta(y))-\widetilde{\rho_{l}}(\alpha\beta(x),\widetilde{\rho_{r}}(\phi(v),y)\\
&=&\widetilde{\rho_{r}}(\rho_{l}(R(\beta(x)),\phi(v)),\beta(y))-\widetilde{\rho_{l}}(\alpha\beta(x),\rho_{r}(\phi(v),R(y))\\
&=&\rho_{r}(\rho_{l}(R(\beta(x)),\phi(v)),R(\beta(y)))-\rho_{l}(R(\alpha\beta(x)),\rho_{r}(\phi(v),R(y))\\
&=&\rho_{r}(\rho_{l}(\beta(R(x)),\phi(v)),\beta(R(y)))-\rho_{l}(\alpha\beta(R(x)),\rho_{r}(\phi(v),R(y))~~~~~~~~~~~~~~~~~~~~~~~~~~~\\
&=&as_{{V}_{\phi,\psi}}(\beta(R(x)),\phi(v),R(y))\\
&=&-as_{{V}_{\phi,\psi}}(\psi(v),\alpha(R(x)),R(y))~~(\textsl{by~(\ref{2})~in ~$V_{\phi,\psi}$})\\
&=&-\rho_{r}(\rho_{r}(\psi(v),\alpha(R(x))),\beta(R(y)))+\rho_{r}(\phi(\psi(v)),(\mu(\alpha(R(x)),R(y))\\
&=&-\rho_{r}(\rho_{r}(\psi(v),R(\alpha(x))),R(\beta(y)))+\rho_{r}(\phi(\psi(v)),R(\mu(R(\alpha(x)),y)+\mu(\alpha(x),R(y)))\\
&=&-\widetilde{\rho_{r}}(\rho_{r}(\psi(v),R(\alpha(x))),\beta(y))+\widetilde{\rho_{r}}(\phi(\psi(v)),\mu(R(\alpha(x)),y)+\mu(\alpha(x),R(y))\\
&=&-\widetilde{\rho_{r}}(\widetilde{\rho_{r}}(\psi(v),\alpha(x)),\beta(y))+\widetilde{\rho_{r}}(\phi(\psi(v)),\mu_{R}(\alpha(x),y)\\
&&-as_{{V}_{\phi,\psi}}(\psi(v),\alpha(x),y),\\
\end{array}$$
similarly, $as_{\widetilde{V}_{\phi,\psi}}(y,\beta(x),\phi(v))=-as_{\widetilde{V}_{\phi,\psi}}(y,\psi(v),\alpha(x)).$

This completes the proof.
\end{proof}
The following Theorem studied by T. Chtioui et al \cite{chtioui2}, shows that a direct sum of a BiHom-alternative algebra and a bimodule over this BiHom-algebra, is still a BiHom-alternative, called the split null extension determined by the given bimodule.
\begin{thm}\cite{chtioui2}\label{smHa}
Let $(A,\mu,\alpha,\beta)$ be a BiHom-alternative algebra and  $(V,\phi,\psi)$  be an $A$-bimodule with the structure maps $\rho_l$ and $\rho_r$. If define on $A\oplus V$ the bilinear maps
$\tilde{\mu}: (A\oplus V)^{\otimes 2}\longrightarrow A\oplus V,$
$\tilde{\mu}(a+m,b+n):=\mu(a,b)+\rho_{l}(a, n)+\rho_{r}(m ,b)$  and the linear maps
$\tilde{\alpha}: A\oplus V\longrightarrow A\oplus V,$ $\tilde{\alpha}(a+m):=\alpha(a)+\phi(m),~$ $\tilde{\beta}: A\oplus V\longrightarrow A\oplus V,$ $\tilde{\beta}(a+m):=\beta(a)+\psi(m).$ Then $E=(A\oplus V,\tilde{\mu},\tilde{\alpha},\tilde{\beta})$ is
a BiHom-alternative algebra.
\end{thm}
\begin{rmk} Consider the  split null extension $A\oplus V$ determined by the BiHom-alternative bimodule $(V,\varphi,\psi)$ for the BiHom-alternative algebra
$(A,\mu,\alpha,\beta)$ in the previous theorem. Write  elements $a+v$ of $A\oplus V$ as $(a,v).$ Then there is an injective homomorphism of BiHom-modules
$i :V\rightarrow A\oplus V $ given by $i(v)=(0,v)$ and a surjective homomorphism of BiHom-modules $\pi : A\oplus V\rightarrow A$ given by $\pi(a,v)=a.$
Moreover, $i(V)$ is a two-sided BiHom-ideal of $A\oplus V$  such that $A\oplus V/i(V)\cong A$. On the other hand, there is a morphism of BiHom-algebras
$\sigma: A\rightarrow A\oplus V$ given by $\sigma(a)=(a,0)$ which is clearly a section of $\pi.$ Hence, we obtain the abelian split exact sequence of
BiHom-alternative algebras and $(V, \alpha_V,\beta_{V})$ is a BiHom-alternative bimodule for $A$ via $\pi.$
 \end{rmk}

\subsection{BiHom-Jordan bimodules}
In this subsection, we study BiHom-Jordan bimodules.
\begin{defn}
Let $(A,\mu,\alpha, \beta)$ be a BiHom-Jordan algebra.
\begin{enumerate}
\item  A right BiHom-Jordan $A$-module is a BiHom-module $(V,\phi,\psi)$   with a right structure map\\ $\rho_r:V\otimes A\longrightarrow V,$ $v\otimes a\mapsto v\cdot a$ such that the following conditions hold:
\begin{eqnarray}
&&\circlearrowleft_{(x,y,z)}\Big\{(\phi\psi^2(v)\cdot\alpha\beta(x)\alpha^2(y))\cdot
\beta\alpha^3(z)-(\phi\psi^2(v)\cdot\beta\alpha^2(z))\cdot\alpha^2\beta(x)
\alpha^3(y)\Big\}=0~~~~~~~~~~\nonumber\\&&
\label{brc2}\\
&& ((\psi^2(v)\cdot\beta\alpha(x))\cdot\beta\alpha^2(y))\cdot\beta\alpha^3(z)
+ ((\psi^2(v)\cdot\beta\alpha(z))\cdot\beta\alpha^2(y))\cdot\beta\alpha^3(x)~~~~~~~~~~~~~~
\nonumber\\
&&+\phi^2\psi^2(v)\cdot(\beta\alpha(x)\alpha^2(z))\alpha^3(y)=
(\phi\psi^2(v)\cdot\beta\alpha^2(x))\cdot\beta\alpha^2(y)\alpha^3(z)\nonumber\\
&&+(\phi\psi^2(v)\cdot\beta\alpha^2(z))\cdot\beta\alpha^2(y)\alpha^3(x)+
(\phi\psi^2(v)\cdot\beta\alpha^2(y))\cdot\alpha^2\beta(x)\alpha^3(z)
\label{brc3}
\end{eqnarray}
 for all $x,y,z\in A$ and $v\in V.$
\item A left BiHom-Jordan $A$-module is a BiHom-module $(V,\phi,\psi)$  with a left structure map\\ $\rho_l:A\otimes V\longrightarrow V,$ $a\otimes v\mapsto a\cdot v$
such that $\psi$ is invertible and the following conditions hold:
\begin{eqnarray}
&&\circlearrowleft_{(x,y,z)}\Big\{\beta^2\alpha^2(z)\cdot(\alpha\beta(x)\alpha^2(y)
\cdot\phi^3(v))-\alpha\beta^2(x)\alpha^2\beta(y)
\cdot(\beta\alpha^2(z)\cdot\phi^3(v))\Big\}=0~~~~~~~~~~~~\nonumber\\
&&\label{blc2}\\
&&\beta^2\alpha^2(z)\cdot(\beta\alpha^2(y)\cdot(\alpha^2(x)
\cdot\psi^{-1}\phi^3(v)))+\beta^2\alpha^2(x)\cdot(\beta\alpha^2(y)
\cdot(\alpha^2(z)\cdot\psi^{-1}\phi^3(v)))~~~~\nonumber\\
&&+((\beta^2(x)\beta\alpha(z))\beta\alpha^2(y))\cdot\phi^3\psi(v)
=\beta^2\alpha(y)\beta\alpha^2(z)\cdot(\beta\alpha^2(x)\cdot\phi^3(v))
\nonumber\\
&&+\beta^2\alpha(y)\beta\alpha^2(x)\cdot(\beta\alpha^2(z)\cdot\phi^3(v))
+\beta^2\alpha(x)\beta\alpha^2(z)\cdot(\beta\alpha^2(y)\cdot\phi^3(v))
\label{blc3}
\end{eqnarray}
for all $x,y,z\in A$ and $v\in V.$
\end{enumerate}
\end{defn}
The following result allows to introduce the notions of right and left special BiHom-Jordan modules.
\begin{thm}\label{thmright}
Let $(A,\mu,\alpha, \beta)$ be a BiHom-Jordan algebra, $(V,\phi,\psi)$ be a BiHom-module and $\rho_r: V\otimes A\rightarrow V,$ $a\otimes v\mapsto v\cdot a,$ be a linear map satisfying
\begin{eqnarray}
\phi\rho_r=\rho_r(\phi\otimes\alpha) \mbox{ and }
\psi\rho_r=\rho_r(\psi\otimes\beta)~~~~~~~~~~~~~~~~~~~~~~~~~~~~~~~~~~~~~~~~~~~~~~~~~~~~~~~~~~
 \label{rc5}\\
\phi(v)\cdot\beta(x)\alpha(y)=(v\cdot \beta(x))\cdot\beta\alpha(y)+(v\cdot \beta(y)\cdot\alpha\beta(x), \mbox{ for all $~x,y\in A$ and $v\in V.~~~~$ }   \label{rc4}
\end{eqnarray}
 Then $(V,\phi,\psi, \rho_r)$ is a right BiHom-Jordan $A$-module called
a right special BiHom-Jordan $A$-module.
\end{thm}
\begin{proof} It suffices to prove $(\ref{brc2})$ and $(\ref{brc3}).$  For all $(x,y)\in A^{\times 2}$ and $v\in V,$ we have:
\begin{eqnarray}
&&\circlearrowleft_{(x,y,z)}(\phi\psi^2(v)\cdot\beta\alpha^2(z))\cdot\alpha^2\beta(x)
\alpha^3(y)=\circlearrowleft_{(x,y,z)}\phi(\psi^2(v)
\cdot\beta\alpha(z))\cdot\beta\alpha^2(x)
\alpha^3(y)\nonumber\\
&&=\circlearrowleft_{(x,y,z)}[(\psi^2(v)\cdot\beta\alpha(z))\cdot\beta\alpha^2(x)]\cdot\beta\alpha^3(y)+
\circlearrowleft_{(x,y,z)}[(\psi^2(v)\cdot\beta\alpha(z))\cdot\beta\alpha^2(y)]\cdot\beta\alpha^3(x) \mbox{  (\textsl{by \ref{rc4}} )}\nonumber\\
&&=\circlearrowleft_{(x,y,z)}(\phi\psi^2(v)\cdot\beta\alpha(z)\alpha^2(x))\cdot\beta\alpha^3(y)-
\circlearrowleft_{(x,y,z)}[(\psi^2(v)\cdot\beta\alpha(x))\cdot\beta\alpha^2(z)]\cdot\beta\alpha^3(y)\nonumber\\
&&\circlearrowleft_{(x,y,z)}[(\psi^2(v)\cdot\beta\alpha(z))\cdot\beta\alpha^2(y)]\cdot\beta\alpha^3(x) \mbox{  ( \textsl{again by \ref{rc4}} )}\nonumber\\
&&=\circlearrowleft_{(x,y,z)}(\phi\psi^2(v)\cdot\beta\alpha(z)\alpha^2(x))\cdot\beta\alpha^3(y)=
\circlearrowleft_{(x,y,z)}\{(\phi\psi^2(v)\cdot\alpha\beta(x)\alpha^2(y))\cdot
\beta\alpha^3(z)\nonumber
\end{eqnarray}
Similarly, we have
\begin{eqnarray}
&&(\phi\psi^2(v)\cdot\beta\alpha^2(x))\cdot\beta\alpha^2(y)\alpha^3(z)
+(\phi\psi^2(v)\cdot\beta\alpha^2(z))\cdot\beta\alpha^2(y)\alpha^3(x)\nonumber\\
&&+(\phi\psi^2(v)\cdot\beta\alpha^2(y))\cdot\alpha^2\beta(x)\alpha^3(z)
=\phi(\psi^2(v)\cdot\beta\alpha(x))\cdot\beta\alpha^2(y)\alpha^3(z)\nonumber\\
&&+\phi(\psi^2(v)\cdot\beta\alpha(z))\cdot\beta\alpha^2(y)\alpha^3(x)\nonumber+
\phi(\psi^2(v)\cdot\beta\alpha(y))\cdot\alpha^2\beta(x)\alpha^3(z)
\mbox{  ( \textsl{ by \ref{rc5}} )}
\nonumber\\
&=&[(\psi^2(v)\cdot\beta\alpha(x))\cdot\beta\alpha^2(y)]\cdot\beta\alpha^3(z)+
[(\psi^2(v)\cdot\beta\alpha(x))\cdot\beta\alpha^2(z)]\cdot\beta\alpha^3(y)\nonumber\\
&&+[(\psi^2(v)\cdot\beta\alpha(z))\cdot\beta\alpha^2(y)]\cdot\beta\alpha^3(x)+
[(\psi^2(v)\cdot\beta\alpha(z))\cdot\beta\alpha^2(x)]\cdot\beta\alpha^3(y)\nonumber\\
&&+[(\psi^2(v)\cdot\beta\alpha(y))\cdot\beta\alpha^2(x)]\cdot\beta\alpha^3(z)+
[(\psi^2(v)\cdot\beta\alpha(y))\cdot\beta\alpha^2(z)]\cdot\beta\alpha^3(x)\mbox{  (  by \ref{rc4} )}\nonumber\\
&=&[(\psi^2(v)\cdot\beta\alpha(x))\cdot\beta\alpha^2(y)]\cdot\beta\alpha^3(z)+
(\phi\psi^2(v)\cdot\beta\alpha(x)\alpha^2(z))\cdot\beta\alpha^3(y)\nonumber\\&&-
[(\psi^2(v)\cdot\beta\alpha(z))\cdot\beta\alpha^2(x)]\cdot\beta\alpha^3(y)+[(\psi^2(v)\cdot\beta\alpha(z))\cdot\beta\alpha^2(y)]\cdot\beta\alpha^3(x)\nonumber\\&&+
[(\psi^2(v)\cdot\beta\alpha(z))\cdot\beta\alpha^2(x)]\cdot\beta\alpha^3(y)+[(\psi^2(v)\cdot\beta\alpha(y))\cdot\beta\alpha^2(x)]\cdot\beta\alpha^3(z)\nonumber\\&&+
[(\psi^2(v)\cdot\beta\alpha(y))\cdot\beta\alpha^2(z)]\cdot\beta\alpha^3(x)\mbox{  ( \textsl{Again  by \ref{rc4} })}\nonumber\\
&=&[(\psi^2(v)\cdot\beta\alpha(x))\cdot\beta\alpha^2(y)]\cdot\beta\alpha^3(z)+
(\phi\psi^2(v)\cdot\beta\alpha(x)\alpha^2(z))\cdot\beta\alpha^3(y)
\nonumber\\&&+[(\psi^2(v)\cdot\beta\alpha(z))\cdot\beta\alpha^2(y)]\cdot\beta\alpha^3(x)
+[(\psi^2(v)\cdot\beta\alpha(y))\cdot\beta\alpha^2(x)]\cdot\beta\alpha^3(z)\nonumber\\&&+
[(\psi^2(v)\cdot\beta\alpha(y))\cdot\beta\alpha^2(z)]\cdot\beta\alpha^3(x)\nonumber\\
&&=[(\psi^2(v)\cdot\beta\alpha(x))\cdot\beta\alpha^2(y)]\cdot\beta\alpha^3(z)+
\phi^2\psi^2(v)\cdot(\beta\alpha(x)\alpha^2(z))\alpha^3(y)\nonumber\\&&-
(\phi\psi^2(v)\cdot\beta\alpha^2(y))\cdot\beta\alpha^2(x)\alpha^3(z)+[(\psi^2(v)\cdot\beta\alpha(z))\cdot\beta\alpha^2(y)]\cdot\beta\alpha^3(x)\nonumber\\&&
+[(\psi^2(v)\cdot\beta\alpha(y))\cdot\beta\alpha^2(x)]\cdot\beta\alpha^3(z)+
[(\psi^2(v)\cdot\beta\alpha(y))\cdot\beta\alpha^2(z)]\cdot\beta\alpha^3(x)\mbox{  ( \textsl{Again  by \ref{rc4} })}\nonumber\\
&&=[(\psi^2(v)\cdot\beta\alpha(x))\cdot\beta\alpha^2(y)]\cdot\beta\alpha^3(z)+
\phi^2\psi^2(v)\cdot(\beta\alpha(x)\alpha^2(z))\alpha^3(y)\nonumber\\&&
+[(\psi^2(v)\cdot\beta\alpha(z))\cdot\beta\alpha^2(y)]\cdot\beta\alpha^3(x)
\mbox{  (\textsl{ Again  by \ref{rc4} })}\nonumber
\end{eqnarray}
It follows that  $(\ref{brc2})$ and $(\ref{brc3})$ hold and therefore $(V,\phi,\psi, \rho_r)$ is a right BiHom-Jordan $A$-module.\end{proof}

 Similarly we prove the following result.
\begin{thm}\label{thmleft}
Let $(A,\mu,\alpha, \beta)$ be a BiHom-Jordan algebra, $(V,\phi,\psi)$ be a BiHom-module such that $\psi$ is invertible and $\rho_l: A\otimes V\rightarrow V,$ $v\otimes a\mapsto a\cdot v,$ be a bilinear map satisfying
\begin{eqnarray}
\phi\rho_l=\rho_l(\alpha\otimes\phi) \mbox{ and }
\psi\rho_l=\rho_l(\beta\otimes\psi)~~~~~~~~~~~~~~~~~~~~~~~~~~~~~~~~~~~~~~~~~~~~~~~~~~~~~~~~~~~~~~~~~
 \label{lc5}\\
\beta(x)\alpha(y)\cdot\psi(v)=\beta\alpha(x)\cdot(\alpha(y)\cdot v)+\beta\alpha(y)\cdot(\alpha(x)\cdot v),   \mbox{ for all $(x,y)\in A^{\times 2}$ and $v\in V.~~~~$ }  \label{lc4}
\end{eqnarray}
 Then $(V,\phi,\psi,\rho_l)$ is a left BiHom-Jordan $A$-module called
a left special BiHom-Jordan $A$-module.\hfill $\square$
\end{thm}
\begin{proof} It suffices to prove $(\ref{blc2})$ and $(\ref{blc3}).$  For all $(x,y)\in A^{\times 2}$ and $v\in V,$ we have:
\begin{eqnarray}
&&\circlearrowleft_{(x,y,z)}\{\alpha\beta^2(x)\alpha^2\beta(y)
\cdot(\beta\alpha^2(z)\cdot\phi^3(v))\}=\circlearrowleft_{(x,y,z)}\{
\alpha\beta^2(x)\alpha^2\beta(y)
\cdot\psi(\alpha^2(z)\cdot\psi^{-1}\phi^3(v))\}\nonumber\\
&&=\circlearrowleft_{(x,y,z)}\{\beta^2\alpha^2(x)\cdot[\alpha^2\beta(y)\cdot
(\alpha^2(z)\cdot\psi^{-1}\phi(v))]\}+
\circlearrowleft_{(x,y,z)}\{\beta^2\alpha^2(y)\cdot[\alpha^2\beta(x)\cdot
(\alpha^2(z)\cdot\psi^{-1}\phi(v))]\}\nonumber\\
&& \mbox{  ( \textsl{  by \ref{lc4} })}\nonumber\\
&&=\circlearrowleft_{(x,y,z)}\{\beta^2\alpha^2(x)\cdot
[\alpha\beta(y)\alpha^2(z)\cdot\phi^3(v)]\}
-\circlearrowleft_{(x,y,z)}\{\beta^2\alpha^2(x)\cdot[\beta\alpha^2(z)
\cdot(\alpha^2(y)\cdot\psi^{-1}\phi(v))]\}
\nonumber\\
&&+\circlearrowleft_{(x,y,z)}\{\beta^2\alpha^2(y)\cdot[\alpha^2\beta(x)\cdot
(\alpha^2(z)\cdot\psi^{-1}\phi(v))]\} \mbox{  ( \textsl{Again  by \ref{lc4} })}\nonumber\\
&&\circlearrowleft_{(x,y,z)}\{\beta^2\alpha^2(x)\cdot[\alpha\beta(y)
\alpha^2(z)\cdot\phi^3(v)]\}
=\circlearrowleft_{(x,y,z)}\{\beta^2\alpha^2(z)\cdot(\alpha\beta(x)\alpha^2(y)
\cdot\phi^3(v))\}\nonumber
\end{eqnarray}
Similarly, we have
\begin{eqnarray}
&&\beta^2\alpha(y)\beta\alpha^2(z)\cdot(\beta\alpha^2(x)\cdot\phi^3(v))
+\beta^2\alpha(y)\beta\alpha^2(x)\cdot(\beta\alpha^2(z)\cdot\phi^3(v))\nonumber\\
&&+\beta^2\alpha(x)\beta\alpha^2(z)\cdot(\beta\alpha^2(y)\cdot\phi^3(v))=
\beta^2\alpha(y)\beta\alpha^2(z)\cdot\psi(\alpha^2(x)\cdot\psi^{-1}\phi^3(v))\nonumber\\
&&+\beta^2\alpha(y)\beta\alpha^2(x)\cdot\psi(\alpha^2(z)\cdot\psi^{-1}\phi^3(v))
+\beta^2\alpha(x)\beta\alpha^2(z)\cdot\psi(\alpha^2(y)\cdot\psi^{-1}\phi^3(v))
\mbox{  ( \textsl{  by \ref{lc5} })}
\nonumber\\
&=&\beta^2\alpha^2(y)\cdot[\beta\alpha^2(z)\cdot(\alpha^2(x)\cdot\psi^{-1}
\phi^3(v))]+\beta^2\alpha^2(z)\cdot[\beta\alpha^2(y)\cdot(\alpha^2(x)
\cdot\psi^{-1}\phi^3(v))]\nonumber\\
&&+\beta^2\alpha^2(y)\cdot[\beta\alpha^2(x)\cdot(\alpha^2(z)\cdot\psi^{-1}
\phi^3(v))]+\beta^2\alpha^2(x)\cdot[\beta\alpha^2(y)\cdot(\alpha^2(z)
\cdot\psi^{-1}\phi^3(v))]\nonumber\\
&&+\beta^2\alpha^2(x)\cdot[\beta\alpha^2(z)\cdot(\alpha^2(y)\cdot\psi^{-1}
\phi^3(v))]+\beta^2\alpha^2(z)\cdot[\beta\alpha^2(x)\cdot(\alpha^2(y)
\cdot\psi^{-1}\phi^3(v))] \mbox{  ( \textsl{  by \ref{lc4} })}\nonumber\\
&=&\beta^2\alpha^2(y)\cdot[\beta\alpha^2(z)\cdot(\alpha^2(x)\cdot\psi^{-1}
\phi^3(v))]+\beta^2\alpha^2(z)\cdot[\beta\alpha(y)\alpha^2(x)\cdot
\phi^3(v))]\nonumber\\
&&-\beta^2\alpha^2(z)\cdot[\beta\alpha^2(x)\cdot(\alpha^2(y)\cdot
\psi^{-1}\phi^3(v))]+\beta^2\alpha^2(y)\cdot[\beta\alpha^2(x)\cdot(\alpha^2(z)
\cdot\psi^{-1}\phi^3(v))]\nonumber\\
&&+\beta^2\alpha^2(x)\cdot[\beta\alpha^2(y)\cdot(\alpha^2(z)
\cdot\psi^{-1}\phi^3(v))]+\beta^2\alpha^2(x)\cdot[\beta\alpha^2(z)\cdot(\alpha^2(y)
\cdot\psi^{-1}\phi^3(v))]\nonumber\\
&&+\beta^2\alpha^2(z)\cdot[\beta\alpha^2(x)\cdot(\alpha^2(y)
\cdot\psi^{-1}\phi^3(v))] \mbox{  ( \textsl{ Again  by \ref{lc4} })}\nonumber\\
&=&\beta^2\alpha^2(y)\cdot[\beta\alpha^2(z)\cdot(\alpha^2(x)\cdot\psi^{-1}
\phi^3(v))]+\beta^2\alpha^2(z)\cdot[\beta\alpha(y)\alpha^2(x)\cdot
\phi^3(v))]\nonumber\\
&&+\beta^2\alpha^2(y)\cdot[\beta\alpha^2(x)\cdot(\alpha^2(z)
\cdot\psi^{-1}\phi^3(v))]+\beta^2\alpha^2(x)\cdot[\beta\alpha^2(y)\cdot(\alpha^2(z)
\cdot\psi^{-1}\phi^3(v))]\nonumber\\
&&+\beta^2\alpha^2(x)\cdot[\beta\alpha^2(z)\cdot(\alpha^2(y)
\cdot\psi^{-1}\phi^3(v))]\nonumber\\
&=&\beta^2\alpha^2(y)\cdot[\beta\alpha^2(z)\cdot(\alpha^2(x)\cdot\psi^{-1}
\phi^3(v))]+\beta^2\alpha^2(z)\cdot[\beta\alpha^2(y)\cdot(\alpha^2(x)\cdot
\psi^{-1}\phi^3(v))]\nonumber\\
&&+\beta^2\alpha^2(z)\cdot[\beta\alpha^2(x)\cdot(\alpha^2(y)\cdot\psi^{-1}
\phi^3(v))]+\beta^2\alpha^2(y)\cdot[\beta\alpha^2(x)\cdot(\alpha^2(z)
\cdot\psi^{-1}\phi^3(v))]\nonumber\\
&&+\beta^2\alpha^2(x)\cdot[\beta\alpha^2(y)\cdot(\alpha^2(z)
\cdot\psi^{-1}\phi^3(v))]+\beta^2\alpha^2(x)\cdot[\beta\alpha^2(z)\cdot(\alpha^2(y)
\cdot\psi^{-1}\phi^3(v))] \mbox{  ( \textsl{ Again  by \ref{lc4} })}\nonumber\\
&=&\beta^2\alpha^2(y)\cdot[\beta\alpha^2(z)\cdot(\alpha^2(x)\cdot\psi^{-1}
\phi^3(v))]+\beta^2\alpha^2(z)\cdot[\beta\alpha^2(y)\cdot(\alpha^2(x)\cdot
\psi^{-1}\phi^3(v))]\nonumber\\
&&+\beta^2\alpha^2(z)\cdot[\beta\alpha^2(x)\cdot(\alpha^2(y)\cdot\psi^{-1}
\phi^3(v))]+\beta^2\alpha^2(y)\cdot[\beta\alpha^2(x)\cdot(\alpha^2(z)
\cdot\psi^{-1}\phi^3(v))]\nonumber\\
&&+\beta^2\alpha^2(x)\cdot[\beta\alpha^2(y)\cdot(\alpha^2(z)
\cdot\psi^{-1}\phi^3(v))]+\beta^2\alpha(x)\beta\alpha^2(z)\cdot(\beta\alpha^2(y)
\cdot\phi^3(v))\nonumber\\
&&-\beta^2\alpha^2(z)\cdot[\beta\alpha^2(x)\cdot(\alpha^2(y)\cdot\psi^{-1}\phi^3(v))] \mbox{  ( \textsl{ Again  by \ref{lc4} })}\nonumber\\
&=&\beta^2\alpha^2(y)\cdot[\beta\alpha^2(z)\cdot(\alpha^2(x)\cdot\psi^{-1}
\phi^3(v))]+\beta^2\alpha^2(z)\cdot[\beta\alpha^2(y)\cdot(\alpha^2(x)\cdot
\psi^{-1}\phi^3(v))]\nonumber\\
&&+\beta^2\alpha^2(y)\cdot[\beta\alpha^2(x)\cdot(\alpha^2(z)
\cdot\psi^{-1}\phi^3(v))]+\beta^2\alpha^2(x)\cdot[\beta\alpha^2(y)\cdot(\alpha^2(z)
\cdot\psi^{-1}\phi^3(v))]\nonumber\\
&&+\beta^2\alpha(x)\beta\alpha^2(z)\cdot(\beta\alpha^2(y)
\cdot\phi^3(v))\nonumber\\
&=&\beta^2\alpha^2(y)\cdot[\beta\alpha^2(z)\cdot(\alpha^2(x)\cdot\psi^{-1}
\phi^3(v))]+\beta^2\alpha^2(z)\cdot[\beta\alpha^2(y)\cdot(\alpha^2(x)\cdot
\psi^{-1}\phi^3(v))]\nonumber\\
&&+\beta^2\alpha^2(y)\cdot[\beta\alpha^2(x)\cdot(\alpha^2(z)
\cdot\psi^{-1}\phi^3(v))]+\beta^2\alpha^2(x)\cdot[\beta\alpha^2(y)\cdot(\alpha^2(z)
\cdot\psi^{-1}\phi^3(v))]\nonumber\\
&&+((\beta^2(x)\beta\alpha(z))\beta\alpha^2(y))
\cdot\psi\phi^3(v)-\beta^2\alpha^2(y)\cdot[\alpha\beta(x)\alpha^2(z)\cdot\phi^3(v)]
\mbox{  ( \textsl{ Again  by \ref{lc4} })}\nonumber\\
&=&\beta^2\alpha^2(z)\cdot[\beta\alpha^2(y)\cdot(\alpha^2(x)\cdot
\psi^{-1}\phi^3(v))]+\beta^2\alpha^2(x)\cdot[\beta\alpha^2(y)\cdot(\alpha^2(z)
\cdot\psi^{-1}\phi^3(v))]\nonumber\\
&&+((\beta^2(x)\beta\alpha(z))\beta\alpha^2(y))
\cdot\psi\phi^3(v) \mbox{  ( \textsl{ Again  by \ref{lc4} })}\nonumber
\end{eqnarray}
It follows that  $(\ref{blc2})$ and $(\ref{blc3})$ hold and therefore $(V,\phi,\psi, \rho_l)$ is a left BiHom-Jordan $A$-module.\end{proof}

It is well known that the plus algebra of any BiHom-alternative algebra is a BiHom-Jordan algebra \cite{chtioui1}.
The next result shows that any left (resp. right) BiHom-alternative module  is also a left (resp. right) module over its plus BiHom-algebra.
\begin{prop}
Let $(A,\mu,\alpha, \beta)$ be a BiHom-altenative algebra and $(V,\phi,\psi)$ be a BiHom-module.
\begin{enumerate}
\item  If $(V,\phi,\psi)$ is a right BiHom-alternative $A$-module with the structure map $\rho_r$
then $(V,\phi,\psi)$ is a right special BiHom-Jordan $A^+$-module with the same
structure map $\rho_r.$
\item If $(V,\phi,\psi)$ is a left BiHom-alternative $A$-module with the structure map $\rho_l$ such that $\psi$ is invertible
then $(V,\phi,\psi)$ is a left special BiHom-Jordan $A^+$-module with the same structure map $\rho_l.$
\end{enumerate}
\end{prop}
\begin{proof}It suffices to prove (\ref{rc4}) and (\ref{lc4}).
\begin{enumerate}
\item If $(V,\phi,\psi)$ is a right BiHom-alternative $A$-module with the structure map $\rho_r,$ then for all $(x,y,v)\in A\times A\times V,$
$as_{V_{\phi,\psi}}(v,\beta(x),\alpha(y))=-as_{V_{\phi,\psi}}(v,\beta(y),\alpha(x)$  i.e.  $\phi(v)\cdot\beta(x)\alpha(y)+\phi(v)\cdot \beta(y)\alpha(x)=(v\cdot \beta(x))\cdot\beta\alpha(y)+(v\cdot \beta(y))\cdot\beta\alpha(x).$ Thus
$\phi(v)\cdot(\beta(x)\ast \alpha(y))=(v\cdot \beta(x))\cdot\beta\alpha(y)+(v\cdot \beta(y))\cdot\beta\alpha(x).$ Therefore $(V,\phi,\psi)$ is a right special BiHom-Jordan $A^+$-module by Theorem \ref{thmright}.
\item If $(V,\phi,\psi)$ is a left BiHom-alternative $A$-module with the structure map $\rho_l,$  then for all
$(x,y,v)\in A\times A\times V,$
$as_{V_{\phi,\psi}}(\beta(x),\alpha(y),v)=-as_{V_{\phi,\psi}}(\beta(y),\alpha(x),v)$  and then $\beta(x)\alpha(y)\cdot\psi(v)+\beta(y)\alpha(x)\psi(v)=
\alpha\beta(x)\cdot(\alpha(y)\cdot v)+\alpha\beta(y)\cdot(\alpha(x)\cdot v).$  Thus
$(\beta(x)\ast \alpha(y)\cdot\psi(v)=\alpha\beta(x)\cdot(\alpha(y)\cdot v)+\alpha\beta(y)\cdot(\alpha(x)\cdot v).$   Therefore  $(V,\phi,\psi)$ is a left special BiHom-Jordan $A^+$-module
by Theorem \ref{thmleft}.
\end{enumerate}
\end{proof}
Now, we give the definition of a BiHom-Jordan  bimodule.
\begin{defn}
Let $(A,\mu,\alpha, \beta)$ be a BiHom-Jordan algebra. \\
A BiHom-Jordan $A$-bimodule is a BiHom-module $(V,\phi,\psi)$  with a left structure map $\rho_l:A\otimes V\longrightarrow V,$ $a\otimes v\mapsto a\cdot v$ and a right structure map $\rho_r:V\otimes A\longrightarrow V,$ $v\otimes a\mapsto v\cdot a$,
such that the following conditions hold:
\begin{eqnarray}
&&\rho_l(\beta\otimes\phi)=\rho_r(\psi\otimes\alpha)\tau_1\label{r3}\\
&& \circlearrowleft_{(x,y,z)}as_{V_{\phi,\psi}}(\mu(\beta^2(x),\alpha\beta(y)),\phi^2\psi(v),\alpha^3(z))=0\label{r4}\\
&&as_{V_{\phi,\psi}}(\psi^2(v)\cdot\beta\alpha(x),\beta\alpha^2(y),\alpha^3(z)) + as_{V_{\phi,\psi}}(\psi^2(v)\cdot\beta\alpha(z),\beta\alpha^2(y),\alpha^3(x)) \nonumber\\
&& + as_{V_{\phi,\psi}}(\mu(\beta^2(x),\beta\alpha(z)),\beta\alpha^2(y),\phi^3(v))=0 \label{r5}
\end{eqnarray}
 for all $x,y,z\in A$ and $v\in V.$
\end{defn}

\begin{rmk}\label{rem1}
\begin{enumerate}
\item If $\alpha=\beta=Id_A$ and $\phi=\psi=Id_V$ then $V$ is reduced to the so-called Jordan module of the Jordan algebra $(A,\mu)$ \cite{Jacob1, Jacob2}.
\item Clearly, a BiHom-Jordan $A$-bimodule is a right
BiHom-Jordan module. Furthermore, it is a left BiHom-Jordan module if $\psi$ is invertible.
\end{enumerate}
\end{rmk}
\begin{ex} \label{e3} Here are some examples of BiHom-Jordan bimodules.
\begin{enumerate}
\item Let $(A,\mu,\alpha,\beta)$ be a BiHom-Jordan algebra. Then $(A,\alpha, \beta)$ is a BiHom-Jordan $A$-bimodule where the structure maps are $\rho_l=\rho_r=\mu.$
More generally, if $B$ is a Two-sided  BiHom-ideal of $(A,\mu,\alpha,\beta),$  then $(B, \alpha, \beta)$ is a BiHom-Jordan $A$-bimodule where the structure maps are $\rho_l(a,x)=\mu(a,x)=\mu(x,a)=\rho_r(x,a)$ for all $(a,x)\in A\times B.$
\item If $(A,\mu)$ is a Jordan algebra and $M$ is a Jordan $A$-bimodule \cite{Jacob2} in the usual sense then $(M,Id_M, Id_M)$ is a BiHom-Jordan $\mathbb{A}$-bimodule where $\mathbb{A}=(A,\mu, Id_A, Id_A)$ is a BiHom-Jordan algebra.
\end{enumerate}
\end{ex}
As BiHom-alternative algebra case, in order to give another example of BiHom-Jordan bimodules, let us consider the following
\begin{defn} An abelian extension of BiHom-Jordan algebras is a short exact sequence of BiHom-Jordan algebras
$$0\longrightarrow(V,\alpha_V,\beta_V)\stackrel{\mbox{i}}{\longrightarrow} (A,\mu_A,\alpha_A,\beta_A)\stackrel{\mbox{$\pi$}}{\longrightarrow} (B,\mu_B,\alpha_B,\beta_B)\longrightarrow 0$$
 where $(V,\alpha_V,\beta_B)$ is a trivial BiHom-Jordan algebra, $i$ and $\pi$ are morphisms of BiHom-algebras. Furthermore,  if there exists a morphism
$s :(B,\mu_B,\alpha_B,\beta_B)\longrightarrow (A,\mu_A,\alpha_A, \beta_A)$ such that $\pi\circ s = id_B$ then the abelian extension is said to be split and $s$ is called a section of $\pi.$
\end{defn}
\begin{ex}
Given an abelian extension as in the previous definition,  the BiHom-module
$(V,\alpha_V, \beta_V)$ inherits a structure of a BiHom-Jordan $B$-bimodule and the actions of the BiHom-algebra $(B,\mu_B,\alpha_B,\beta_B)$  on $V$ are as follows. For any $x\in B,$ there exist $\tilde{x}\in A$ such that $x=\pi(\tilde{x}).$ Let $x$ acts on $v\in V$ by $x\cdot v:=\mu_A(\tilde{x},i(v))$ and $v\cdot x:=\mu_A(i(v),\tilde{x}).$
These are well-defined, as another lift $\tilde{x'}$
of $x$ is written $\tilde{x'}=\tilde{x}+v'$ for some $v'\in V$ and thus
$x\cdot v=\mu_A(\tilde{x},i(v))=\mu_A(\tilde{x'},i(v))$ and $v\cdot x=\mu_A(i(v),\tilde{x})=\mu_A(i(v),\tilde{x'})$ because $V$ is trivial. The actions property follow from the BiHom-Jordan identity.
In case these actions of $B$ on $V$  are trivial, one speaks of a central extension.
\end{ex}
The next result shows that a special left and right BiHom-Jordan module has a BiHom-Jordan bimodule structure  under a specific condition.
\begin{thm}\label{ModPlus}
Let $(A,\mu,\alpha,\beta)$ be a regular BiHom-Jordan algebra and $(V,\phi,\psi)$ be both a left  and a right special BiHom-Jordan $A$-module with the structure
maps $\rho_1$ and $\rho_2$ respectively such that $\phi$ is invertible and the BiHom-associativity (or operator BiHom-commutativity) condition holds
\begin{eqnarray}
\rho_2\circ(\rho_1\otimes\beta)=\rho_1\circ(\alpha\otimes\rho_2) \label{associat}
\end{eqnarray}
 Define the bilinear maps $\rho_l: A\otimes V\longrightarrow V$ and $\rho_r: V\otimes A\longrightarrow V$ by
\begin{eqnarray}
\rho_l=\rho_1+\rho_2(\psi\phi^{-1}\otimes\alpha\beta^{-1})\circ\tau_1 \mbox{ and } \rho_r=\rho_1(\beta\alpha^{-1}\otimes\phi\psi^{-1})\circ\tau_2+\rho_2 \label{StrucPlus}
\end{eqnarray}
Then $(V, \phi,\psi,\rho_l,\rho_r)$ is a BiHom-Jordan $A$-bimodule.
\end{thm}
\begin{proof}It is clear that $\rho_l$ and $\rho_r$ are structures maps. To prove (\ref{r3}), (\ref{r4}) and (\ref{r5}), let put  $\rho_l(x,v)=x\cdot v+\psi\phi^{-1}(v)\cdot\alpha\beta^{-1}(x)$ and $\rho_r(v,x)=\beta\alpha^{-1}(x)\cdot\phi\psi^{-1}(v)+v\cdot x$ for all $(x,v)\in A\times V.$
\\
First, for all $(x,v)\in A\times V$ we obtain $\rho_l(\beta(x),\phi(v))=\beta(x)\cdot\phi(v)+\psi(v)\cdot\alpha(x)
=\rho_r(\psi(v),\alpha(x))$ and thus, we get $(\ref{r3}).$ Next, let $(x,y,z)\in A^3$ and $v\in V.$ Then
\begin{eqnarray}
&&as_{V_{\phi,\psi}}(\beta^2(x)\alpha\beta(y),\phi^2\psi(v),\alpha^3(z))\nonumber\\
&=&\rho_r(\rho_l(\beta^2(x)\alpha\beta(y),\phi^2\psi(v)),\beta\alpha^3(z))
-\rho_l(\alpha\beta^2(x)\alpha^2\beta(y),\rho_r(\phi^2\psi(v),\alpha^3(z)))
\nonumber\\
&=&\rho_r(\beta^2(x)\alpha\beta(y)\cdot\phi^2\psi(v),\beta\alpha^3(z)
+\rho_r(\psi^2\phi(v)\cdot\alpha\beta(x)\alpha^2(y),\beta\alpha^3(z))
\nonumber\\
&&-\rho_l(\alpha\beta^2(x)\alpha^2\beta(y)\phi^2\psi(v)\cdot\alpha^3(z))
-\rho_l(\alpha\beta^2(x)\alpha^2\beta(y),\beta\alpha^2(z)\cdot\phi^3(v))
\mbox{ ( by (\ref{StrucPlus}) )}\nonumber\\
&=&[\beta^2(x)\alpha\beta(y)\cdot\phi^2\psi(v)]\cdot\beta\alpha^3(z)
+\beta^2\alpha^2(z)\cdot[\alpha\beta(x)\alpha^2(y)\cdot\phi^3(v)]\nonumber\\
&&+[\psi^2\phi(v)\cdot\alpha\beta(x)\alpha^2(y)]\cdot\beta\alpha^3(z)
+\beta^2\alpha^2(z)\cdot[\phi^2\psi(v)\cdot\alpha^2(x)\alpha^3\beta^{-1}(y)]\nonumber\\
&&-\alpha\beta^2(x)\alpha^2\beta(y)\cdot(\phi^2\psi(v)\cdot\alpha^3(z))-
(\phi\psi^2(v)\cdot\beta\alpha^2(z))\cdot\alpha^2\beta(x)\alpha^3(y)\nonumber\\
&&-\alpha\beta^2(x)\alpha^2\beta(y)\cdot(\beta\alpha^2(z)\cdot\phi^3(v))-
(\beta^2\alpha(z)\cdot\psi\phi^2(v))\cdot\alpha^2\beta(x)\alpha^3(y)
\nonumber\\
&& \mbox{ ( again by (\ref{StrucPlus}) )}\nonumber\\
&=&\{[\psi^2\phi(v)\cdot\alpha\beta(x)\alpha^2(y)]\cdot\beta\alpha^3(z)
 -(\phi\psi^2(v)\cdot\beta\alpha^2(z))\cdot\alpha^2\beta(x)\alpha^3(y) \}
 \nonumber\\
&&+\{\beta^2\alpha^2(z)\cdot[\alpha\beta(x)\alpha^2(y)\cdot\phi^3(v)]-
 \alpha\beta^2(x)\alpha^2\beta(y)\cdot(\beta\alpha^2(z)\cdot\phi^3(v))\}
 \nonumber\\
 &&\{[\beta^2(x)\alpha\beta(y)\cdot\phi^2\psi(v)]\cdot\beta\alpha^3(z)
 -\alpha\beta^2(x)\alpha^2\beta(y)\cdot(\phi^2\psi(v)\cdot\alpha^3(z)) \}
 \nonumber\\
 &&\{\beta^2\alpha^2(z)\cdot[\phi^2\psi(v)\cdot\alpha^2(x)\alpha^3\beta^{-1}(y)]-
 (\beta^2\alpha(z)\cdot\psi\phi^2(v))\cdot\alpha^2\beta(x)\alpha^3(y)\}
 \nonumber\\
 && \mbox{(rearranging terms )}\nonumber\\
 &=&\{[\psi^2\phi(v)\cdot\alpha\beta(x)\alpha^2(y)]\cdot\beta\alpha^3(z)
 -(\phi\psi^2(v)\cdot\beta\alpha^2(z))\cdot\alpha^2\beta(x)\alpha^3(y) \}
 \nonumber\\
 &&+\{\beta^2\alpha^2(z)\cdot[\alpha\beta(x)\alpha^2(y)\cdot\phi^3(v)]-
 \alpha\beta^2(x)\alpha^2\beta(y)\cdot(\beta\alpha^2(z)\cdot\phi^3(v))\}
 \nonumber\\
 && \mbox{( using (\ref{associat})  )}\nonumber
\end{eqnarray}
Hence
\begin{eqnarray}
&&\circlearrowleft_{(x,y,z)}as_{V_{\phi,\psi}}(\beta^2(x)\alpha\beta(y),\phi^2\psi(v),\alpha^3(z))\nonumber\\
&=&\circlearrowleft_{(x,y,z)}\{[\psi^2\phi(v)\cdot\alpha\beta(x)\alpha^2(y)]\cdot\beta\alpha^3(z)
 -(\phi\psi^2(v)\cdot\beta\alpha^2(z))\cdot\alpha^2\beta(x)\alpha^3(y) \}
 \nonumber\\
 &&+\circlearrowleft_{(x,y,z)}\{\beta^2\alpha^2(z)\cdot[\alpha\beta(x)\alpha^2(y)\cdot\phi^3(v)]-
 \alpha\beta^2(x)\alpha^2\beta(y)\cdot(\beta\alpha^2(z)\cdot\phi^3(v))\}
 \nonumber\\
 &=&0 \mbox{ ( by (\ref{brc2}) and (\ref{blc2}) ).}\nonumber
\end{eqnarray}
Thus, we get (\ref{r4}).
\\
Now, to prove (\ref{r5}), let compute each of its three terms.
\begin{eqnarray}
&&as_{V_{\phi,\psi}}(\rho_r(\psi^2(v),\beta\alpha(x)),\beta\alpha^2(y),\alpha^3(z))\nonumber\\
&=&\rho_r(\rho_r(\beta^2(x)\cdot\phi\psi(v),\beta\alpha^2(y)),\beta\alpha^3(z))+\rho_r(\rho_r(\psi^2(v)\cdot\beta\alpha(x),\beta\alpha^2(y)),\beta\alpha^3(z))\nonumber\\
&&-\rho_r(\alpha\beta^2(x)\cdot\phi^2\psi(v),\beta\alpha^2(y)\alpha^3(z)
-\rho_r(\phi\psi^2(v)\cdot\beta\alpha^2(x),\beta\alpha^2(y)\alpha^3(z)
\nonumber\\
&=&\rho_r(\beta^2\alpha(y)\cdot(\alpha\beta(x)\cdot\phi^2(v))+
(\beta^2(x)\cdot\phi\psi(v))\cdot\beta\alpha^2(y),\beta\alpha^3(z))
\nonumber\\
&&+\rho_r(\beta^2\alpha(y)\cdot(\phi\psi(v)\cdot\alpha^2(x))+(\psi^2(v)
\cdot\beta\alpha(x))\cdot\beta\alpha^2(y),\beta\alpha^3(z))\nonumber\\
&&-\beta^2\alpha(y)\beta\alpha^2(z)\cdot(\alpha^2\beta(x)\cdot\phi^3(v))
-(\alpha\beta^2(x)\cdot\phi^2\psi(v))\cdot\beta\alpha^2(y)\alpha^3(z)
\nonumber\\
&&-\beta^2\alpha(y)\beta\alpha^2(z)\cdot(\phi^2\psi(v)\cdot\alpha^3(x))
-(\phi\psi^2(v)\cdot\beta\alpha^2(x))\cdot\beta\alpha^2(y)\alpha^3(z)
\nonumber
\end{eqnarray}
\begin{eqnarray}
&=&  \beta^2\alpha^2(z)\cdot[\beta\alpha^2(y)\cdot(\alpha^2(x)\cdot
\phi^3\psi^{-1}(v))]+\beta^2\alpha^2(z)\cdot[(\alpha\beta(x)\cdot\phi^3(v))\cdot\alpha^3(y)]             \nonumber\\
&&[\beta^2\alpha(y)\cdot(\alpha\beta(x)\cdot\phi^3(v))]\cdot\beta\alpha^3(z)
+[(\beta^2(x)\cdot\phi\psi(v))\cdot\beta\alpha^2(y)]\cdot\beta\alpha^3(z)
\nonumber\\
&&+\beta^2\alpha^2(z)\cdot[\beta\alpha^2(y)\cdot(\phi^2(v)\cdot\alpha^3
\beta^{-1}(x))]+\beta^2\alpha^2(z)\cdot[(\phi\psi(v)\cdot\alpha^2(x))\cdot\alpha^3(y)]
\nonumber\\
&&+ [\beta^2\alpha(y)\cdot(\phi\psi(v)\cdot\alpha^2(x))]\cdot\beta\alpha^3(z)+
[(\psi^2(v)\cdot\beta\alpha(x))\cdot\beta\alpha^2(y)]\cdot\beta\alpha^3(z) \nonumber\\
&&-\beta^2\alpha(y)\beta\alpha^2(z)\cdot(\alpha^2\beta(x)\cdot\phi^3(v))
-\underbrace{(\alpha\beta^2(x)\cdot\phi^2\psi(v))\cdot\beta\alpha^2(y)
\alpha^3(z)}_{A}
\nonumber\\
&&-\underbrace{\beta^2\alpha(y)\beta\alpha^2(z)\cdot(\phi^2\psi(v)\cdot
\alpha^3(x))}_{B}
-(\phi\psi^2(v)\cdot\beta\alpha^2(x))\cdot\beta\alpha^2(y)\alpha^3(z)
\nonumber
\end{eqnarray}
Observe that
\begin{eqnarray}
A&&=\phi(\beta^2(x)\cdot\phi\psi(v))\cdot\beta(\alpha^2(x))
\alpha(\alpha^2(x))\nonumber\\
&&=[(\beta^2(x)\cdot\psi\phi(v))\cdot\beta\alpha^2(y)]\cdot\beta\alpha^3(z)
+[(\beta^2(x)\cdot\phi\psi(v))\cdot\beta\alpha^2(z)]\cdot\beta\alpha^3(y)
\mbox{ ( by (\ref{rc4}) ) }\nonumber\\
&&=[(\beta^2(x)\cdot\psi\phi(v))\cdot\beta\alpha^2(y)]\cdot\beta\alpha^3(z)
+[\alpha\beta^2(x)\cdot(\phi\psi(v)\cdot\alpha^2(z)]\cdot\beta\alpha^3(y)
\mbox{  ( by (\ref{associat}) )}\nonumber\\
&&=[(\beta^2(x)\cdot\psi\phi(v))\cdot\beta\alpha^2(y)]\cdot\beta\alpha^3(z)
+\alpha^2\beta^2(x)\cdot[(\phi\psi(v)\cdot\alpha^2(z))\cdot\alpha^3(y)]
\mbox{  (again  by (\ref{associat}) )}\nonumber\\
B&&=\beta^2\alpha(y)\beta\alpha^2(z)\cdot\psi(\phi^2(v)\cdot\beta^{-1}\alpha^3(x))
\nonumber\\
&&=\beta^2\alpha^2(y)\cdot[\beta\alpha^2(z)\cdot(\phi^2(v)\cdot
\beta^{-1}\alpha^3(x)]+\beta^2\alpha^2(z)\cdot[\beta\alpha^2(y)
\cdot(\phi^2(v)\cdot\beta^{-1}\alpha^3(x))]\mbox{ ( by (\ref{lc4}) )}\nonumber\\
&&=\beta^2\alpha^2(z)\cdot[(\beta\alpha(z)\cdot\phi^2(v))\cdot\alpha^3(x)]
+\beta^2\alpha^2(z)\cdot[\beta\alpha^2(y)
\cdot(\phi^2(v)\cdot\beta^{-1}\alpha^3(x))]\mbox{ ( by (\ref{associat}) )}\nonumber\\
&&=[\alpha\beta^2(y)\cdot(\beta\alpha(z)\cdot\phi^2(v))]\cdot\beta\alpha^3(x)
+\beta^2\alpha^2(z)\cdot[\beta\alpha^2(y)
\cdot(\phi^2(v)\cdot\beta^{-1}\alpha^3(x))]\mbox{ ( again by (\ref{associat}) )}\nonumber
\end{eqnarray}
Thus, if we replace $A$ and $B,$ we get:
\begin{eqnarray}
&&as_{V_{\phi,\psi}}(\rho_r(\psi^2(v),\beta\alpha(x)),\beta\alpha^2(y),\alpha^3(z))\nonumber\\
&=&\beta^2\alpha^2(z)\cdot[\beta\alpha^2(y)\cdot(\alpha^2(x)\cdot
\phi^3\psi^{-1}(v))]+\beta^2\alpha^2(z)\cdot[(\alpha\beta(x)\cdot\phi^3(v))\cdot\alpha^3(y)]             \nonumber\\
&&[\beta^2\alpha(y)\cdot(\alpha\beta(x)\cdot\phi^3(v))]\cdot\beta\alpha^3(z)
+[(\beta^2(x)\cdot\phi\psi(v))\cdot\beta\alpha^2(y)]\cdot\beta\alpha^3(z)
\nonumber\\
&&+\beta^2\alpha^2(z)\cdot[\beta\alpha^2(y)\cdot(\phi^2(v)\cdot\alpha^3
\beta^{-1}(x))]+\beta^2\alpha^2(z)\cdot[(\phi\psi(v)\cdot\alpha^2(x))\cdot\alpha^3(y)]
\nonumber\\
&&+ [\beta^2\alpha(y)\cdot(\phi\psi(v)\cdot\alpha^2(x))]\cdot\beta\alpha^3(z)+
[(\psi^2(v)\cdot\beta\alpha(x))\cdot\beta\alpha^2(y)]\cdot\beta\alpha^3(z) \nonumber\\
&&-\beta^2\alpha(y)\beta\alpha^2(z)\cdot(\alpha^2\beta(x)\cdot\phi^3(v))
-[(\beta^2(x)\cdot\psi\phi(v))\cdot\beta\alpha^2(y)]\cdot\beta\alpha^3(z)
\nonumber\\
&&-\alpha^2\beta^2(x)\cdot[(\phi\psi(v)\cdot\alpha^2(z))\cdot\alpha^3(y)]
-[\alpha\beta^2(y)\cdot(\beta\alpha(z)\cdot\phi^2(v))]\cdot\beta\alpha^3(x)
\nonumber\\
&&-\beta^2\alpha^2(z)\cdot[\beta\alpha^2(y)
\cdot(\phi^2(v)\cdot\beta^{-1}\alpha^3(x))]
-(\phi\psi^2(v)\cdot\beta\alpha^2(x))\cdot\beta\alpha^2(y)\alpha^3(z)
\label{as1}
\end{eqnarray}
If we permute $x$ and $z$ in (\ref{as1}), we get
\begin{eqnarray}
&&as_{V_{\phi,\psi}}(\rho_r(\psi^2(v),\beta\alpha(z)),\beta\alpha^2(y),\alpha^3(x))\nonumber\\
&=&\beta^2\alpha^2(x)\cdot[\beta\alpha^2(y)\cdot(\alpha^2(z)\cdot
\phi^3\psi^{-1}(v))]+\beta^2\alpha^2(x)\cdot[(\alpha\beta(z)\cdot\phi^3(v))\cdot\alpha^3(y)]             \nonumber\\
&&[\beta^2\alpha(y)\cdot(\alpha\beta(z)\cdot\phi^3(v))]\cdot\beta\alpha^3(x)
+[(\beta^2(z)\cdot\phi\psi(v))\cdot\beta\alpha^2(y)]\cdot\beta\alpha^3(x)
\nonumber\\
&&+\beta^2\alpha^2(x)\cdot[\beta\alpha^2(y)\cdot(\phi^2(v)\cdot\alpha^3
\beta^{-1}(z))]+\beta^2\alpha^2(x)\cdot[(\phi\psi(v)\cdot\alpha^2(z))\cdot\alpha^3(y)]
\nonumber\\
&&+ [\beta^2\alpha(y)\cdot(\phi\psi(v)\cdot\alpha^2(z))]\cdot\beta\alpha^3(x)+
[(\psi^2(v)\cdot\beta\alpha(z))\cdot\beta\alpha^2(y)]\cdot\beta\alpha^3(x) \nonumber\\
&&-\beta^2\alpha(y)\beta\alpha^2(x)\cdot(\alpha^2\beta(z)\cdot\phi^3(v))
-[(\beta^2(z)\cdot\psi\phi(v))\cdot\beta\alpha^2(y)]\cdot\beta\alpha^3(x)
\nonumber\\
&&-\alpha^2\beta^2(z)\cdot[(\phi\psi(v)\cdot\alpha^2(x))\cdot\alpha^3(y)]
-[\alpha\beta^2(y)\cdot(\beta\alpha(x)\cdot\phi^2(v))]\cdot\beta\alpha^3(z)
\nonumber\\
&&-\beta^2\alpha^2(x)\cdot[\beta\alpha^2(y)
\cdot(\phi^2(v)\cdot\beta^{-1}\alpha^3(z))]
-(\phi\psi^2(v)\cdot\beta\alpha^2(z))\cdot\beta\alpha^2(y)\alpha^3(x)
\label{as2}\nonumber
\end{eqnarray}
Again, we have
\begin{eqnarray}
&&as_{V_{\phi,\psi}}(\beta^2(x)\beta\alpha(z),\beta\alpha^2(y),\phi^3(v))\nonumber\\
&=&\rho_l((\beta^2(x)\alpha\beta(z))\beta\alpha^2(y),\psi\phi^3(v))
-\rho_l(\alpha\beta^2(x)\beta\alpha^2(z),\rho_l(\beta\alpha^2(y),\phi^3(v))
\nonumber\\
&=&[(\beta^2(x)\alpha\beta(z))\beta\alpha^2(y)]\cdot\psi\phi^3(v)
+\psi^2\phi^2(v)\cdot[(\alpha\beta(x)\alpha^2(z))\alpha^3(y)]\nonumber\\
&&-\alpha\beta^2(x)\beta\alpha^2(z)\cdot(\beta\alpha^2(y)\cdot\phi^3(v))
-\underbrace{\alpha\beta^2(x)\beta\alpha^2(z)\cdot(
\psi\phi^2(v)\cdot\alpha^3(y))}_{C}
\nonumber\\
&&-\underbrace{(\beta^2\alpha(y)\cdot\psi\phi^2(v))\cdot\alpha^2\beta(x)
\alpha^3(z)}_{D}
-(\psi^2\phi(v)\cdot\beta\alpha^2(y))\cdot\alpha^2\beta(x)\alpha^3(z)
\nonumber\\
&&\mbox{ ( by a direct computation )}\nonumber
\end{eqnarray}
Observe that
\begin{eqnarray}
C &&=\alpha\beta^2(x)\beta\alpha^2(z)\cdot
\psi(\phi^2(v)\cdot\beta^{-1}\alpha^3(y))\nonumber\\
&&=\alpha^2\beta^2(x)\cdot[\beta\alpha^2(z)\cdot(\phi^2(v)\cdot
\beta^{-1}\alpha^3(y))]+\alpha^2\beta^2(z)\cdot[\beta\alpha^2(x)\cdot(\phi^2(v)\cdot
\beta^{-1}\alpha^3(y))] \mbox{  ( by (\ref{lc4} )}\nonumber\\
&&=\alpha^2\beta^2(x)\cdot[(\beta\alpha(z)\cdot\phi^2(v))\cdot\alpha^3(y)]
+\alpha^2\beta^2(z)\cdot[(\beta\alpha(x)\cdot\phi^2(v))\cdot\alpha^3(y)]
\mbox{  ( by (\ref{associat} )}\nonumber\\
D&&=\phi(\beta^2(y)\cdot\psi\phi(v))\cdot\alpha^2\beta(x)
\alpha^3(z)\nonumber\\
&&=[(\beta^2(y)\cdot\psi\phi(v))\cdot\beta\alpha^2(x)]\cdot\beta\alpha^3(z)
+[(\beta^2(y)\cdot\psi\phi(v))\cdot\beta\alpha^2(z)]\cdot\beta\alpha^3(x)
\mbox{  ( by (\ref{rc4} )}\nonumber\\
&&=[\alpha\beta^2(y)\cdot(\psi\phi(v)\cdot\alpha^2(x))]\cdot\beta\alpha^3(z)
+[\alpha\beta^2(y)\cdot(\psi\phi(v)\cdot\alpha^2(z))]\cdot\beta\alpha^3(x)
 \mbox{  ( by (\ref{associat} )}\nonumber
\end{eqnarray}
Hense, if we replace $C$ and $D$, we obtain
\begin{eqnarray}
&&as_{V_{\phi,\psi}}(\beta^2(x)\beta\alpha(z),\beta\alpha^2(y),\phi^3(v))\nonumber\\
&=&[(\beta^2(x)\alpha\beta(z))\beta\alpha^2(y)]\cdot\psi\phi^3(v)
+\psi^2\phi^2(v)\cdot[(\alpha\beta(x)\alpha^2(z))\alpha^3(y)]\nonumber\\
&&-\alpha\beta^2(x)\beta\alpha^2(z)\cdot(\beta\alpha^2(y)\cdot\phi^3(v))
-\alpha^2\beta^2(x)\cdot[(\beta\alpha(z)\cdot\phi^2(v))\cdot\alpha^3(y)]
\nonumber\\
&&-\alpha^2\beta^2(z)\cdot[(\beta\alpha(x)\cdot\phi^2(v))\cdot\alpha^3(y)]
-[\alpha\beta^2(y)\cdot(\psi\phi(v)\cdot\alpha^2(x))]\cdot\beta\alpha^3(z)
\nonumber\\
&&-[\alpha\beta^2(y)\cdot(\psi\phi(v)\cdot\alpha^2(z))]\cdot\beta\alpha^3(x)
-(\psi^2\phi(v)\cdot\beta\alpha^2(y))\cdot\alpha^2\beta(x)\alpha^3(z)
\nonumber
\end{eqnarray}
 Finally, using (\ref{brc3}), (\ref{blc3}) and (\ref{associat}), we obtain
\begin{eqnarray}
&&as_{V_{\phi,\psi}}(\psi^2(v)\cdot\beta\alpha(x),\beta\alpha^2(y),\alpha^3(z)) + as_{V_{\phi,\psi}}(\psi^2(v)\cdot\beta\alpha(z),\beta\alpha^2(y),\alpha^3(x)) \nonumber\\
&& + as_{V_{\phi,\psi}}(\mu(\beta^2(x),\beta\alpha(z)),\beta\alpha^2(y),\phi^3(v))=0\nonumber
\end{eqnarray}
\end{proof}
The following result will be used below. It gives a relation between BiHom-associative modules  and special BiHom-Jordan modules.
\begin{lem}\label{HomAssModule}
Let $(A,\mu,\alpha,\beta)$ be a BiHom-associative algebra and $(V,\phi,\psi)$ be a BiHom-module.
\begin{enumerate}
\item If $(V,\phi,\psi)$ is a right BiHom-associative $A$-module with the structure maps $\rho_r$
then $(V,\phi,\psi)$ is a right special BiHom-Jordan $A^+$-module with the same structure map $\rho_r.$
\item  If $(V,\phi,\psi)$ is a left BiHom-associative $A$-module with the structure maps $\rho_l$ such that $\psi$ is invertible
then, $(V,\phi,\psi)$ is a left special BiHom-Jordan $A^+$-module with the same structure map $\rho_l.$
\end{enumerate}
\end{lem}
\begin{proof} It also suffices to prove (\ref{rc4}) and (\ref{lc4}). To start, let denote the product $\mu'$ in Proposition \ref{assoJord} by $\ast.$\end{proof}
\begin{enumerate}
\item If $(V,\phi,\psi)$ is a right BiHom-associative $A$-module with the structure map $\rho_r$ then for all $(x,y,v)\in A\times A\times V,$
$\phi(v)\cdot(\beta(x)\ast\beta(y))=\phi(v)\cdot(\beta(x)\alpha(y)+\
\beta(y)\alpha(x))=(v\cdot\beta(x))\cdot\beta\alpha(x)+(v\cdot\beta(y))\cdot
\beta\alpha(x)$
where the last equality holds by (\ref{11}). Then $(V,\phi,\psi)$ is a right special BiHom-Jordan  $A^+$-module by Theorem \ref{thmright}.
\item If $(V,\phi,\psi)$ is a left BiHom-associative $A$-module with the structure map $\rho_l$  then for all
$(x,y,v)\in A\times A\times V,$
$(\beta(x)\ast\beta(y))\dot\psi(v)=(\beta(x)\alpha(y)+\beta(y)\alpha(x))\cdot
\psi(v)=\alpha\beta(x)\cdot(\alpha(y)\cdot v)+\alpha\beta(y)\cdot(\alpha(x)\cdot v)$
 where the last equality holds by (\ref{00}). Then  $(V,\phi,\psi)$ is a left special BiHom-Jordan $A^+$-module by Theorem \ref{thmleft}.\hfill $\square$
\end{enumerate}
Now, we prove that a BiHom-associative module  gives rise to a BiHom-Jordan module for its plus BiHom-algebra.
\begin{prop}\label{bhamHJ}
Let $(A,\mu,\alpha_A)$ be a BiHom-associative algebra and $(V, \rho_1, \rho_2,\phi,\psi)$  be a BiHom-associative
$A$-bimodule such that $\phi$ and $\psi$ are inversible. Then $(V, \rho_1, \rho_2,\phi,\psi)$ is a BiHom-Jordan $A^+$-bimodule where $\rho_l$ and $\rho_r$ are defined as in (\ref{StrucPlus}).
\end{prop}
\begin{proof}The proof follows from Lemma \ref{HomAssModule}  and  Theorem \ref{ModPlus}.\end{proof}
\begin{prop}\label{HJB-HJB}
Let $(A,\mu,\alpha, \beta)$ be a BiHom-Jordan algebra and  $V_{\phi,\psi}=(V,\phi, \psi)$  be a BiHom-Jordan $A$-bimodule with the structure maps $\rho_l$ and $\rho_r$.  Then for each $n\in\mathbb{N},$ the maps
\begin{eqnarray}
\rho_l^{(n)}=\rho_l\circ(\alpha^n\otimes Id_V)\label{nmj1}\\
\rho_r^{(n)}=\rho_r\circ(Id_V\otimes \beta^n)\label{nmj2}
\end{eqnarray}
give the BiHom-module $(V,\phi, \psi)$ the structure of a BiHom-Jordan $A$-bimodule that we denote
by $V_{\phi,\psi}^{(n)}$
\end{prop}
\begin{proof} Since the structure map $\rho_l$ is a morphism of BiHom-modules, we get:
$$\begin{array}{lllll}
\phi\rho_l^{(n)}&=&\phi\rho_l\circ(\alpha^n\otimes Id_V)
\mbox{ \textsl{( by (\ref{nmj1}) )}}\nonumber\\
&=&\rho_l\circ(\alpha^{n+1}\otimes\phi)\\ 
&=&\rho_l\circ(\alpha^n\otimes Id_V)\circ(\alpha\otimes\phi)\nonumber\\
&=&\rho_l^{(n)}\circ(\alpha\otimes\phi)\nonumber
\end{array}$$
 Similarly, we get that $\psi\rho_l^{(n)}=\rho_l^{(n)}\circ(\beta\otimes\psi),\
 \phi\rho_r^{(n)}=\rho_r^{(n)}\circ(\phi\otimes\alpha)$ and $\psi\rho_r^{(n)}=\rho_r^{(n)}\circ(\psi\otimes\beta)$  and  that (\ref{r3}) holds for $V_{\phi,\psi}^{(n)}.$ Thus $\rho_l^{(n)}$ and $\rho_r^{(n)}$ are morphisms of BiHom-modules.

Next, the condition (\ref{r4}) in $V_{\phi,\psi}^{(n)}$ follows from the one, in $V_{\phi,\psi}.$ Now,  we compute
$$\begin{array}{llllllll}
&&\circlearrowleft_{(x,y,z)}as_{V^{(n)}_{\phi,\psi}}(\beta^2(x)\alpha\beta(y),\phi^2\psi(v),\alpha^3(z))\nonumber\\
&=&\circlearrowleft_{(x,y,z)}\Big\{\rho_r^{(n)}(\rho_l^{(n)}(\beta^2(x)\alpha\beta(y),\phi^2\psi(v)),\beta\alpha^3(z))\nonumber\\
&&-\rho_l^{(n)}(\alpha\beta^2(x)\alpha^2\beta(y),\rho_r^{(n)}(\phi^2\psi(v),\alpha^3(z)))\Big\}\nonumber\\&=&\circlearrowleft_{(x,y,z)}\Big\{\rho_r^{(n)}(\rho_l(\alpha^n\beta^2(x)\alpha^{n+1}\beta(y),\phi^2\psi(v)),\beta\alpha^3(z))\nonumber\\&&-
\rho_l^{(n)}(\alpha\beta^2(x)\alpha^2\beta(y),\rho_r(\phi^2\psi(v),\alpha^3\beta^n(z)))\Big\}\nonumber\\
&=&\circlearrowleft_{(x,y,z)}\Big\{\rho_r(\rho_l(\alpha^n\beta^2(x)\alpha^{n+1}\beta(y),\phi^2\psi(v)),\beta\alpha^3\beta^n(z))\nonumber\\&&-
\rho_l(\alpha\beta^2\alpha^n(x)\alpha^2\beta\alpha^n(y),\rho_r(\phi^2\psi(v),\alpha^3\beta^n(z)))\Big\}\nonumber\\
&=&\circlearrowleft_{(x,y,z)}\Big\{\rho_r(\rho_l(\beta^2\alpha^n(x)
\alpha\beta\alpha^{n}(y),\phi^2\psi(v)),\beta(\alpha^3\beta^n(z)))\nonumber\\&&-
\rho_l(\alpha(\beta^2\alpha^n(x)\alpha\beta\alpha^n(y)),\rho_r(\phi^2\psi(v),\alpha^3\beta^n(z)))\Big\}\nonumber\\
&=&\circlearrowleft_{(x,y,z)}as_{V_{\phi,\psi}}(\beta^2(\alpha^n(x))\alpha\beta(\alpha^n(y)),\phi^2\psi(v),\alpha^3(\beta^n(z)))\nonumber\\
&=&\circlearrowleft_{(\alpha^n(x),\alpha^n(y),\beta^n(z))}as_{V_{\phi,\psi}}(\beta^2(\alpha^n(x))\alpha\beta(\alpha^n(y)),\phi^2\psi(v),\alpha^3(\beta^n(z)))\nonumber\\
&=&0 \mbox{ ( \textsl{by (\ref{r4}}) in $V_{\phi,\psi}$ )}\nonumber
\end{array}$$
Then we get (\ref{r4}) for $V_{\phi,\psi}^{(n)}.$ Finally remarking that
\begin{eqnarray}
&&as_{V^{(n)}_{\phi,\psi}}(\rho_r^{(n)}(\psi^2(v),\alpha\beta(x)),\beta\alpha^2(y),\alpha^3(z))\nonumber\\
&=&as_{V^{(n)}_{\phi,\psi}}(\psi^2(v)\cdot\alpha\beta^{n+1}(x),\beta\alpha^2(y),\alpha^3(z))\nonumber\\
&=&\rho_r^{(n)}(\rho_r^{(n)}(\psi^2(v)\cdot\alpha\beta^{n+1}(x),\beta\alpha^2(y)),\beta\alpha^3(z))\nonumber\\&&-
\rho_r^{(n)}(\phi\psi^2(v)\cdot\alpha^2\beta^{n+1}(x),\mu(\beta\alpha^2(y),\alpha^3(z)))\nonumber\\
&=&\rho_r(\rho_r(\psi^2(v)\cdot\alpha\beta^{n+1}(x),\beta\alpha^2\beta^n(y)),\beta\alpha^3\beta^n(z))\nonumber\\&&-
\rho_r(\phi\psi^2(v)\cdot\alpha^2\beta^{n+1}(x),\mu(\beta\alpha^2\beta^n(y),\alpha^3\beta^n(z)))\nonumber\\
&=&\rho_r(\rho_r(\psi^2(v)\cdot\beta\alpha(\beta^n(x)),\beta\alpha^2(\beta^n(y))),\beta(\alpha^3\beta^n(z)))\nonumber\\&&-
\rho_r(\phi(\psi^2(v)\cdot\beta\alpha\beta^n(x),\mu(\beta\alpha^2\beta^n(y),\alpha^3\beta^n(z)))\nonumber\\
&=&as_{V_{\phi,\psi}}(\psi^2(v)\cdot\beta\alpha(\beta^n(x)),\beta\alpha^2(\beta^n(y)),\alpha^3(\beta^n(z)))\nonumber
\end{eqnarray}
and similarly
\begin{eqnarray}
as_{V^{(n)}_{\phi,\psi}}(\rho_r^{(n)}(\psi^2(v),\alpha\beta(z)),\beta\alpha^2(y),\alpha^3(x))
&=&as_{V_{\phi,\psi}}(\psi^2(v)\cdot\beta\alpha(\beta^n(z)),\beta\alpha^2(\beta^n(y)),\alpha^3(\beta^n(x))),\nonumber\\
as_{V^{(n)}_{\phi,\psi}}(\beta^2(x)\beta\alpha(z),\beta\alpha^2(y),\phi^3(v))
&=&as_{V_{\phi,\psi}}(\beta^2(\beta^n(x))\beta\alpha(\beta^n(z)),\beta\alpha^2(\beta^n(y)),\phi^3(v))~.\nonumber
\end{eqnarray}
Then the relation (\ref{r5}) in $V_{\phi,\psi}^{(n)}$  follows from the one in $V_{\phi,\psi}.$
We conclude that $V_{\phi,\psi}^{(n)}$ is a BiHom-Jordan $A$-bimodule.\end{proof}
The following result  says that Jordan bimodules can  be deformed into BiHom-Jordan bimodules via an endomorphism.
\begin{thm}\label{HJB-JB}
Let $(A,\mu)$ be a Jordan algebra, $V$ be a Jordan $A$-bimodule with the structure maps
$\rho_l$ and $\rho_r$ and $\alpha,\beta$ be  endomorphisms of the Jordan algebra $A$
and $\phi, \psi$ be  linear self-maps of $V$ such that $\phi\circ\rho_l=\rho_l\circ(\alpha\otimes\phi),$
$\phi\circ\rho_r=\rho_r\circ(\phi\otimes\alpha),$
$\psi\circ\rho_l=\rho_l\circ(\beta\otimes\psi)$ and
$\psi\circ\rho_r=\rho_r\circ(\psi\otimes\beta)$.\\
Write $A_{\alpha,\beta}$ for the BiHom-Jordan algebra $(A,\mu_{\alpha,\beta}=\mu(\alpha\otimes\beta),\alpha, \beta)$ and
$V_{\phi,\psi}$ for the BiHom-module $(V,\phi,\psi).$ Then the maps:
\begin{eqnarray}
\tilde{\rho_l}=\rho_l(\alpha\otimes\psi) \mbox{ and }
\tilde{\rho_r}=\rho_r(\phi\otimes\beta)
\end{eqnarray}
give the BiHom-module $V_{\phi,\psi}$ the structure of a BiHom-Jordan $A_{\alpha,\beta}$-bimodule.
\end{thm}
\begin{proof}It is easy to prove that the relation (\ref{r3}) for $V_{\phi,\psi}$ holds and both maps $\tilde{\rho_l},$ $\tilde{\rho_r}$ are morphisms. Next, we have
\begin{eqnarray}
&& \circlearrowleft_{(x,y,z)}as_{V_{\phi,\psi}}(\mu_{\alpha,\beta}(\beta^2(x),\alpha\beta(y)),
\phi^2\psi(v),\alpha^3(z))=\circlearrowleft_{(x,y,z)}as_{V_{\phi,\psi}}(\alpha\beta^2(x)\alpha\beta^2(y)),
\phi^2\psi(v),\alpha^3(z))\nonumber\\
&&=\circlearrowleft_{(x,y,z)}\{\tilde{\rho_r}(\tilde{\rho_l}(\alpha\beta^2(x)\alpha\beta^2(y),\phi^2\psi(v)),\beta\alpha^3(z))-
\tilde{\rho_l}(\alpha^2\beta^2(x)\alpha^2\beta^2(y),\tilde{\rho_r}(\phi^2\psi(v),\alpha^3(z)))  \}\nonumber\\
&&=\circlearrowleft_{(x,y,z)}\{\tilde{\rho_r}(\rho_l(\alpha^2\beta^2(x)\alpha^2\beta^2(y),\phi^2\psi^2(v)),\beta\alpha^3(z))-
\tilde{\rho_l}(\alpha^2\beta^2(x)\alpha^2\beta^2(y),\rho_r(\phi^3\psi(v),\beta\alpha^3(z)))  \}\nonumber\\
&&=\circlearrowleft_{(x,y,z)}\{\rho_r(\rho_l(\alpha^3\beta^2(x)\alpha^3\beta^2(y),\phi^3\psi^2(v)),\beta^2\alpha^3(z))-
\tilde{\rho_l}(\alpha^3\beta^2(x)\alpha^3\beta^2(y),\rho_r(\phi^3\psi^2(v),\beta^2\alpha^3(z)))  \}\nonumber\\
&&=\circlearrowleft_{(x,y,z)} as_V(\alpha^3\beta^2(x)\alpha^3\beta^2(y),\phi^3\psi^2(v),\alpha^3\beta^2(z))\nonumber\\
&&=\circlearrowleft_{(\alpha^3\beta^2(x),\alpha^3\beta^2(y),\alpha^3\beta^2(z))} as_V(\alpha^3\beta^2(x)\alpha^3\beta^2(y),\phi^3\psi^2(v),\alpha^3\beta^2(z))
=0 \mbox { ( \textsl{by  (\ref{r4}}) in $V.$}\nonumber
\end{eqnarray}
Therefore, we get (\ref{r4}) for $V_{\phi,\psi}$. Finally, we have
\begin{eqnarray}
&& as_{V_{\phi,\psi}}(\tilde{\rho_r}(\psi^2(v),
\beta\alpha(x)),\beta\alpha^2(y),\alpha^3(z))\nonumber\\
&&= as_{V_{\phi,\psi}}(\phi\psi^2(v)\cdot\beta^2\alpha(x),\beta\alpha^2(y),\alpha^3(z))\nonumber\\
&&=\tilde{\rho_r}(\tilde{\rho_r}(\phi\psi^2(v)\cdot\beta^2\alpha(x),\beta\alpha^2(y)),\beta\alpha^3(z))-\tilde{\rho_r}(\phi^2\psi^2(v)\cdot\beta^2\alpha^2(x),\mu_{\alpha,\beta}(\beta\alpha^2(y),\alpha^3(z)))\nonumber\\
&&=\tilde{\rho_r}(\rho_r(\phi^2\psi^2(v)\cdot\beta^2\alpha^2(x),\beta^2\alpha^2(y)),\beta\alpha^3(z))-\tilde{\rho_r}(\phi^2\psi^2(v)\cdot\beta^2\alpha^2(x),\beta\alpha^3(y)\beta\alpha^3(z))\nonumber\\
&&=\rho_r(\rho_r(\phi^3\psi^2(v)\cdot\beta^2\alpha^3(x),\beta^2\alpha^3(y)),\beta^2\alpha^3(z))-\rho_r(\phi^3\psi^2(v)\cdot\beta^2\alpha^3(x),\beta^2\alpha^3(y)\beta^2\alpha^3(z))\nonumber\\
&&= as_V(\phi^3\psi^2(v)\cdot\beta^2\alpha^3(x),\beta^2\alpha^3(y)),\beta^2\alpha^3(z))\nonumber
\end{eqnarray}
and similarly
\begin{eqnarray}
as_{V_{\phi,\psi}}(\tilde{\rho_r}(\psi^2(v),
\beta\alpha(z)),\beta\alpha^2(y),\alpha^3(x))=
as_V(\phi^3\psi^2(v)\cdot\beta^2\alpha^3(z),\beta^2\alpha^3(y)),\beta^2\alpha^3(x))\nonumber\\
as_{V_{\phi,\psi}}(\mu_{\alpha,\beta}(\beta^2(x),\beta\alpha(z)),\beta\alpha^2(y),\phi^3(v))=
as_V(\mu(\beta^2\alpha^3(x),\beta^2\alpha^3(z)),\beta^2\alpha^3(y),\phi^3\psi^2(v)).\nonumber
\end{eqnarray}
Hence, we obtain (\ref{r5}) in $V_{\phi,\psi}$ as follows
\begin{eqnarray}
&&as_{V_{\phi,\psi}}(\tilde{\rho_r}(\psi^2(v),
\beta\alpha(x)),\beta\alpha^2(y),\alpha^3(z))+as_{V_{\phi,\psi}}(\tilde{\rho_r}(\psi^2(v),
\beta\alpha(z)),\beta\alpha^2(y),\alpha^3(x))\nonumber\\
&&+as_{V_{\phi,\psi}}(\mu_{\alpha,\beta}(\beta^2(x),\beta\alpha(z)),\beta\alpha^2(y),\phi^3(v))=as_V(\phi^3\psi^2(v)\cdot\beta^2\alpha^3(x),\beta^2\alpha^3(y)),\beta^2\alpha^3(z))\nonumber\\
&&+as_V(\phi^3\psi^2(v)\cdot\beta^2\alpha^3(z),\beta^2\alpha^3(y)),\beta^2\alpha^3(x))
+as_V(\mu(\beta^2\alpha^3(x),\beta^2\alpha^3(z)),\beta^2\alpha^3(y),\phi^3\psi^2(v))\nonumber\\
&&=0 \mbox{ (\textsl{ by (\ref{r5}}) in $V).$}\nonumber
\end{eqnarray}
Therefore the BiHom-module $V_{\phi,\psi}$ has a BiHom-Jordan $A_{\alpha,\beta}$-bimodule structure.
\end{proof}
\begin{cor}
Let $(A,\mu)$ be a Jordan algebra, $V$ be a Jordan $A$-bimodule with the structure maps
$\rho_l$ and $\rho_r$, $\alpha,\ \beta$ be  endomorphisms of the Jordan algebra $A$
and $\phi, \psi$ be  linear self-maps of $V$ such that $\phi\circ\rho_l=\rho_l\circ(\alpha\otimes\phi),$
$\phi\circ\rho_r=\rho_r\circ(\phi\otimes\alpha),$
$\psi\circ\rho_l=\rho_l\circ(\beta\otimes\psi)$ and
$\psi\circ\rho_r=\rho_r\circ(\psi\otimes\beta).$\\
Write $A_{\alpha,\beta}$ for the BiHom-Jordan algebra $(A,\mu_{\alpha,\beta}=\mu(\alpha\otimes\beta),\alpha, \beta)$ and
$V_{\phi,\psi}$ for the BiHom-module $(V,\phi,\psi).$ Then the maps:
\begin{eqnarray}
\tilde{\rho}_l^{(n)}=\rho_l\circ(\alpha^{n+1}\otimes\psi) \mbox{ and }
\tilde{\rho}_r^{(n)}=\rho_r\circ(\phi\otimes\beta^{n+1})
\end{eqnarray}
give the BiHom-module $V_{\phi,\psi}$ the structure of a BiHom-Jordan $A_{\alpha,\beta}$-bimodule for
each $n\in\mathbb{N}$.
\end{cor}
\begin{proof}The proof follows from Proposition \ref{HJB-HJB} and Theorem \ref{HJB-JB}.
\end{proof}
Similarly to BiHom-alternative algebras, the split null extension, determined by the given bimodule over a BiHom-Jordan algebra, is constructed as follows:
\begin{thm}
Let $(A,\mu,\alpha, \beta)$ be a BiHom-Jordan algebra and  $(V,\phi, \psi)$  be a BiHom-Jordan $A$-bimodule with the structure maps $\rho_l$ and $\rho_r$. Then $(A\oplus V,\tilde{\mu},\tilde{\alpha}, \tilde{\beta})$ is
a BiHom-Jordan algebra where \\
$\tilde{\mu}: (A\oplus V)^{\otimes 2}\longrightarrow A\oplus V,~$
$\tilde{\mu}(a+m,b+n):=ab+a\cdot n+m\cdot b$  and
$\tilde{\alpha}, \tilde{\beta}: A\oplus V\longrightarrow A\oplus V,$ $\tilde{\alpha}(a+m):=\alpha(a)+\phi(m)$  and $\tilde{\beta}(a+m):=\beta(a)+\psi(m)$
\end{thm}
\begin{proof} First, the BiHom-commutativity of $\tilde{\mu}$ follows from the one of $\mu$ and the condition (\ref{r3}).
Next, the multiplicativity of $\tilde{\alpha}$ and $\tilde{\beta}$ with respect to $\tilde{\mu}$ follows from those of $\alpha$ and
$\beta$ with respect to $\mu$ and the fact that $\rho_l$ and $\rho_r$ are morphisms of BiHom-modules.
Finally, we prove the BiHom-Jordan identity for $E=A\oplus V$ as it follows
\begin{eqnarray}
&&as_{\tilde{\alpha},\tilde{\beta}}(\tilde{\mu}(\tilde{\beta}^2(x+m),\tilde{\alpha}\tilde{\beta}(x+m)),\tilde{\alpha}^2\tilde{\beta}(y+n),\tilde{\alpha}^2(x+m))\nonumber\\
&=&\tilde{\mu}(\tilde{ \mu}(\tilde{\mu}(\tilde{\beta}^2(x+m),\tilde{\alpha}\tilde{\beta}(x+m)),\tilde{\alpha}^2\bar{\beta}(y+n)),\tilde{\beta}\tilde{\alpha}^3(x+m))\nonumber\\
&&-\tilde{\mu}(\tilde{\alpha}\tilde{\mu}(\tilde{\beta}^2(x+m),\tilde{\alpha}\tilde{\beta}(x+m)),\tilde{\mu}(\tilde{\alpha}^2\tilde{\beta}(y+n),\tilde{\alpha}^3(x+m)))\nonumber\\
&=&\tilde{\mu}\Big[\tilde{\mu}(\beta^2(x)\alpha\beta(x)+\beta^2(x)\cdot\phi\psi(m)+\psi^2(m)\cdot\alpha
\beta(x),\alpha^2\beta(y)+\phi^2\psi(n)),\beta\alpha^3(x)+\psi\phi^3(m)\Big]
\nonumber\\
&&-\tilde{\mu}\Big[\tilde{\alpha}(\beta^2(x)\alpha\beta(x)+\beta^2(x)\cdot\phi\psi(m)+\psi^2(m)\cdot
\alpha\beta(x)),\alpha^2\beta(y)\alpha^3(x)+\alpha^2\beta(y)\cdot\phi^3(m)\nonumber\\&&
+\phi^2\psi(n)\cdot\alpha^3(x)\Big]\nonumber\\
&=&\tilde{\mu}\Big[(\beta^2(x)\alpha\beta(x))\alpha^2\beta(y)+\beta^2(x)\alpha\beta(x)
\cdot\phi^2\psi(n)+(\beta^2(x)\cdot\phi\psi(m))\cdot\alpha^2\beta(y)\nonumber\\&&+
(\psi^2(m)\cdot\alpha\beta(x))\cdot\alpha^2\beta(y),\beta\alpha^3(x)+\psi\phi^3(m)\Big]
-\tilde{\mu}\Big[\alpha\beta^2(x)\alpha^2\beta(x)+\alpha\beta^2(x)\cdot\phi^2\psi(m)\nonumber\\&&+
\phi\psi^2(m)\cdot\alpha^2\beta(x), \alpha^2\beta(y)\alpha^3(x)+\alpha^2\beta(y)\cdot\phi^3(m)
+\phi^2\psi(n)\cdot\alpha^3(x)\Big]\nonumber\\
&=&[(\beta^2(x)\alpha\beta(x))\alpha^2\beta(y)]\beta\alpha^3(x)+
[(\beta^2(x)\alpha\beta(x))\alpha^2\beta(y)]\cdot\psi\phi^3(m)\nonumber\\&&
+[\beta^2(x)\alpha\beta(x)\cdot\phi^2\psi(n)]\cdot\beta\alpha^3(x)+[(\beta^2(x)\cdot\phi\psi(m))\cdot\alpha^2\beta(y)]\cdot\beta\alpha^3(x)\nonumber\\&&
+[(\psi^2(m)\cdot\alpha\beta(x))\cdot\alpha^2\beta(y)]\cdot\beta\alpha^3(x)
-(\alpha\beta^2(x)\alpha^2\beta(x))\alpha^2\beta(y)\alpha^3(x)\nonumber\\&&-\alpha\beta^2(x)\alpha^2\beta(x)\cdot(\alpha^2\beta(y)\cdot\phi^3(m))
-\alpha\beta^2(x)\alpha^2\beta(x)\cdot(\phi^2\psi(n)\cdot\alpha^3(x))\nonumber\\&&
-(\alpha\beta^2(x)\cdot\phi^2\psi(m))\cdot\alpha^2\beta(y)\alpha^3(x)-(\phi\psi^2(m)\cdot\alpha^2\beta(x))\cdot\alpha^2\beta(y)\alpha^3(x)
\nonumber\\
&=& as_{\alpha,\beta}(\beta^2(x)\alpha\beta(x),\alpha^2\beta(y),\alpha^3(x))+
as_{\phi,\psi}(\beta^2(x)\alpha\beta(x),\alpha^2\beta(y),\phi^3(m))\nonumber\\&&
+as_{\phi,\psi}(\beta^2(x)\alpha\beta(x),\phi^2\psi(n),\alpha^3(x))
+as_{\phi,\psi}(\beta^2(x)\cdot\phi\psi(m),\alpha^2\beta(y),\alpha^3(x))\nonumber\\
&&+as_{\phi,\psi}(\psi^2(m)\cdot\alpha\beta(x),\alpha^2\beta(y),\alpha^3(x))\nonumber\\&=& as_{\alpha,\beta}(\beta^2(x)\alpha\beta(x),\alpha^2\beta(y),\alpha^3(x))+
as_{\phi,\psi}(\beta^2(x)\alpha\beta(x),\phi^2\psi(n),\alpha^3(x))
\nonumber\\
&&+\{ as_{\phi,\psi}(\psi^2(m)\cdot\alpha\beta(x),\alpha^2\beta(y),\alpha^3(x))+as_{\phi,\psi}(\psi^2(m)\cdot\alpha\beta(x),\alpha^2\beta(y),\alpha^3(x))\nonumber
\end{eqnarray}
\begin{eqnarray}
&&+as_{\phi,\psi}(\beta^2(x)\alpha\beta(x),\alpha^2\beta(y),\phi^3(m))\} \mbox{  ( \textsl{rearranging terms and using (\ref{r3} )} )}\nonumber\\
&=&0 \mbox{ \textsl{(using respectively the BiHom-Jordan identity, (\ref{r4}) and (\ref{r5})  ))}}\nonumber
\end{eqnarray}
  We conclude then that $(A\oplus V,\tilde{\mu},\tilde{\alpha}, \tilde{\beta})$ is a BiHom-Jordan algebra.\end{proof}

 Similarly as BiHom-alternative algebra case, let give the following:
 \begin{rmk} Consider the  split null extension $A\oplus V$ determined by the BiHom-Jordan bimodule $(V,\phi,\psi)$ for the BiHom-Jordan algebra
$(A,\mu,\alpha,\beta)$ in the previous theorem. Write  elements $a+v$ of $A\oplus V$ as $(a,v).$ Then there is an injective homomorphism of BiHom-modules
$i :V\rightarrow A\oplus V $ given by $i(v)=(0,v)$ and a surjective homomorphism of BiHom-modules $\pi : A\oplus V\rightarrow A$ given by $\pi(a,v)=a.$
Moreover, $i(V)$ is a two-sided BiHom-ideal of $A\oplus V$  such that $A\oplus V/i(V)\cong A$. On the other hand, there is a morphism of BiHom-algebras
$\sigma: A\rightarrow A\oplus V$ given by $\sigma(a)=(a,0)$ which is clearly a section of $\pi.$ Hence, we obtain the abelian split exact sequence of
BiHom-Jordan algebras and $(V, \alpha_V,\beta_{V})$ is a BiHom-Jordan bimodule for $A$ via $\pi.$
 \end{rmk}

\end{document}